\documentclass[article,hidelinks,onefignum,onetabnum]{siamart220329}
\usepackage[utf8]{inputenc}

\usepackage{thm-restate}
\usepackage{lipsum}
\usepackage{amsfonts}
\usepackage{graphicx}
\usepackage{epstopdf}
\usepackage{algorithmic}
\usepackage{placeins}
\usepackage{adjustbox}
\usepackage{mathtools}
\usepackage{makecell}
\usepackage{amssymb}

\definecolor{darkred}{rgb}{.6,0,0}
\definecolor{darkblue}{rgb}{0,0,.7}
\definecolor{darkgreen}{rgb}{0,.7,0}
\definecolor{darkbrown}{rgb}{0.8,0.4,0.4}

\newcommand{\mt}[1]{{\color{black}{#1}}}

\ifpdf
  \DeclareGraphicsExtensions{.eps,.pdf,.png,.jpg}
\else
  \DeclareGraphicsExtensions{.eps}
\fi

\newcommand\blfootnote[1]{%
  \begingroup
  \renewcommand\thefootnote{}\footnote{#1}%
  \addtocounter{footnote}{-1}%
  \endgroup
}

\newsiamthm{lemma}{Lemma}
\newsiamthm{corollary}{Corollary}
\newsiamthm{proposition}{Proposition}
\newsiamthm{definition}{Definition}

\newsiamremark{remark}{Remark}
\newsiamremark{hypothesis}{Hypothesis}
\crefname{hypothesis}{Hypothesis}{Hypotheses}
\newsiamthm{claim}{Claim}

\headers{Learning Homogenization for Elliptic Operators}{K. Bhattacharya, N. B. Kovachki, A. Rajan, A. M. Stuart, M. Trautner}

\title{Learning Homogenization for Elliptic Operators\thanks{Submitted to the editors July 7, 2023.
\funding{The work of KB, AR, and AMS was sponsored, in part, by the Army
Research Laboratory under Cooperative Agreement W911NF22-2-0120, and KB is also sponsored
under W911NF-22-1-0269. The work of AMS is supported by a Department of Defense Vannevar Bush Faculty Fellowship,
and by the SciAI Center, funded by the Office of Naval Research (ONR), under Grant Number N00014-23-1-2729.
The work of MT is funded by the Department of Energy Computational Science Graduate Fellowship 
under award DE-SC0021110. NBK is grateful to the NVIDIA Corporation for support through full-time employment. \mt{The computations presented here were conducted in the Resnick High Performance Computing Center, a facility supported by Resnick Sustainability Institute at the California Institute of Technology.} \mt{The authors are also grateful to 
Samuel Lanthaler and Endre S\"uli for essential discussions leading to the proof of Lemma \ref{lem:BVLinfty_L2compact}.
}}}}

\author{Kaushik Bhattacharya\thanks{California Institute of Technology, Pasadena, CA 
  (\email{bhatta@caltech.edu}, \email{arajan@caltech.edu}, \email{astuart@caltech.edu}, \email{trautner@caltech.edu}).}
\and Nikola B. Kovachki\thanks{NVIDIA, Santa Clara, CA (nkovachki@nvidia.com).}
\and Aakila Rajan\footnotemark[2]
\and Andrew M. Stuart\footnotemark[2]
\and Margaret Trautner\footnotemark[2]}

\usepackage{amsopn}

\newcommand{\Td}{\mathbb{T}^d}
\newcommand{\R}{\mathbb{R}}
\newcommand{\Rd}{\mathbb{R}^d}
\newcommand{\Rdd}{\mathbb{R}^{d \times d}}
\newcommand{\PD}{\mathsf{PD}_{\alpha, \beta}}
\newcommand{\dy}{\: \mathsf{d}y}
\newcommand{\dx}{\: \mathsf{d}x}

\newcommand{\A}{\mathcal{A}}
\newcommand{\G}{\mathcal{G}}
\newcommand{\supp}{\text{supp }}

\newcommand{\chione}{\chi^{(1)}}
\newcommand{\chitwo}{\chi^{(2)}}
\newcommand{\fone}{f^{(1)}}
\newcommand{\ftwo}{f^{(2)}}
\newcommand{\Aone}{A^{(1)}}
\newcommand{\Atwo}{A^{(2)}}

\newcommand{\eps}{\epsilon}
\newcommand{\p}{\partial}
\newcommand{\BV}{\mathsf{BV}}
\newcommand{\Lip}{\mathsf{Lip}}
\newcommand{\Tdtwo}{\mathbb{Y}^d}
\usepackage[colorinlistoftodos,prependcaption,textsize=tiny]{todonotes}

\begin{document}

\maketitle

\begin{abstract}

    Multiscale partial differential equations (PDEs) arise in various applications, and several schemes have been developed to solve them efficiently. Homogenization theory is a powerful methodology that eliminates the small-scale dependence, resulting in simplified equations that are computationally tractable while accurately
    predicting the macroscopic response. In the field of continuum mechanics, homogenization is crucial for deriving constitutive laws that incorporate microscale physics in order to formulate balance laws for the macroscopic quantities of interest. However, obtaining homogenized constitutive laws is often challenging as they do not in general have an analytic form and can exhibit phenomena not present on the microscale. In response, data-driven learning of the constitutive law has been proposed as appropriate for this task. However, a major challenge in data-driven learning approaches for this problem has remained unexplored: the impact of discontinuities and corner interfaces in the underlying material. These discontinuities in the coefficients affect the smoothness of the solutions of the underlying equations. Given the prevalence of discontinuous materials in continuum mechanics applications, it is important to address the challenge of learning in this context; in particular, to develop underpinning theory that establishes the reliability of data-driven methods in this scientific domain. The paper addresses this unexplored challenge by investigating the learnability of homogenized constitutive laws for elliptic operators in the presence of such complexities. Approximation theory is presented, and numerical experiments are performed which validate the theory in the context of learning the
    solution operator defined by the
    cell problem arising in homogenization for elliptic PDEs.
    
\end{abstract}

\section{Introduction} \blfootnote{\mt{All code for this work may be found at \href{https://github.com/mtrautner/LearningHomogenization/}{github.com/mtrautner/LearningHomogenization/}.}} Homogenization theory is a well-established
methodology that aims to eliminate fast-scale dependence in partial differential equations (PDEs) to obtain  homogenized PDEs which produce a good approximate solution of the problem with fast scales while being more computationally tractable. In continuum mechanics, this methodology is of great practical importance as the constitutive laws derived from physical principles are governed by material behavior at small scales, but the quantities of interest are often relevant on larger scales. These homogenized constitutive laws often do not have a
closed analytic form and may have new features not present in the microscale laws. Consequently, there has been a recent surge of interest in employing data-driven methods to learn homogenized constitutive laws. 

The goal of this paper is to study the learnability of homogenized
constitutive laws in the context of one of the canonical model problems of homogenization, namely the divergence form elliptic PDE. One significant challenge in applications of homogenization in material
science arises from the presence of discontinuities and corner interfaces in the underlying material. This leads to a lack of smoothness in the coefficients and solutions of the associated equations, a phenomenon extensively studied in numerical methods for PDEs. Addressing this challenge in the context of learning remains largely unexplored and is the focus of our work. We develop underlying theory and provide accompanying numerical studies to address learnability in this context.

In Subsection \ref{ssec:PF} we establish the mathematical framework and notation for the problem of interest, state the three main
contributions of the paper, and overview the contents of each section 
of the paper. In Subsection \ref{ssec:LR} we provide a detailed literature review. Subsection \ref{ssec:SE} states the stability estimates that are key for the approximation theory developed in the paper and discusses the remainder of the paper in the context of these estimates.
\label{sec:I}
\subsection{Problem Formulation}
\label{ssec:PF}

Consider the following linear multiscale elliptic equation on a bounded domain $\Omega \subset \Rd$:
\begin{subequations}
\label{eqn:ms_ellip}
\begin{align}
    -\nabla_x \cdot\left(A^{\eps}\nabla_x u^{\eps}\right) & = f \quad x \in \Omega, \\
    u^\eps & = 0 \quad x \in \p\Omega.
\end{align}
\end{subequations}
Here $A^{\eps}(x) = A\left(\frac{x}{\eps}\right)$ for $A(\cdot)$ which is $1$-periodic and positive definite: $A: \Td \to \R^{d \times d}_{{\rm sym}, \succ 0}$, \mt{a condition which holds throughout this work. 
Assume further that $f \in L^2(\Omega; \R)$ and has no microscale variation with respect to $x/\eps$.}

Our focus is on linking
this multiscale problem to the homogenized form of equation \eqref{eqn:ms_ellip}, which is
\begin{subequations}\label{eqn:homogenized}
    \begin{align}
        -\nabla_x \cdot\left(\overline{A}\nabla_x u \right) & = f \quad x \in \Omega, \\
        u & = 0 \quad x \in \p\Omega,
    \end{align}
\end{subequations}
where $\overline{A}$ is given by 
\begin{equation}
\label{eq:cell}
    \overline{A} = \int_{\Td}\left(A(y) + A(y) \nabla \chi(y)^T\right) \; \dy,
\end{equation}
and $\chi:\Td \to \Rd$ solves the cell problem
\begin{equation}\label{eqn:cellprob}
    -\nabla \cdot(\nabla \chi A) = \nabla \cdot A, \quad \chi \;\text{is }1\text{-periodic}.
\end{equation}
\mt{All of the preceding PDEs are to be interpreted  as holding in the weak sense.}
For $0 < \eps \ll 1$, the solution $u^{\eps}$ of \eqref{eqn:ms_ellip} is approximated by the solution $u$ of \eqref{eqn:homogenized}, and the error converges to zero as $\eps \to 0$ in various 
topologies \cite{bensoussan2011asymptotic,blanc2023homogenization,pavliotis2008multiscale}. 

We assume that
\[\|A\|_{L^\infty} := \sup_{y \in \Td} |A(y)|_F <\infty\]
where \(|\cdot|_F\) is the Frobenius norm. 
Hence $A \in L^{\infty}\left(\Td;\Rdd\right)$ and $A^{\eps} \in L^{\infty}\left(\Omega;\Rdd\right).$ Similarly, for \(A \in L^2(\Td;\Rdd)\), we define
\[\|A\|_{L^2}^2 := \int_{\Td} |A(y)|_F^2 \dy.\]
\mt{Also, for given $\beta \ge \alpha>0$, we}
define the following subset
of  1-periodic, positive-definite, symmetric matrix fields in
$L^{\infty}\left(\Td;\Rdd\right)$ by
\[\PD = \{A \in L^\infty(\Td;\Rdd) : \; \forall (y,\xi) \in \Td \times \Rd, \; \alpha |\xi|^2 \leq \langle \xi, A(y) \xi \rangle \leq \beta |\xi|^2\}.\]
\mt{For open set $\Omega \subset \Rd$, we denote the \textit{variation} of a function $u \in L^1_{\text{loc}}(\Omega)$ by 
\begin{equation*}
    V(u, \Omega) = \sup\Big\{\sum_{i = 1}^d \int_{\Omega}\frac{\p\Phi_i}{\p x_i}u \dx: \; \Phi \in C_0^{\infty}(\Omega;\Rd), \; \|\Phi\|_{L^{\infty}(\Omega;\Rd)} \leq 1\Big\}
\end{equation*}
and the set of functions of bounded variation on $\Td$ as 
\[\BV = \{u \in L^1(\Td): \; V(u,\Td) < \infty \}.\]
For further information on $\BV$, we refer to $\cite{leoni2017first}$.
Finally, we often work in the Sobolev space $H^1$ 
restricted to spatially
mean-zero periodic functions, denoted
$$\dot{H}^1: = \Bigl\{v \in W^{1,2}(\Td)\; \Big| \; v \text{ is } 1\text{-periodic}, \int_{\Td} v \; \dy = 0\Bigr\};$$
the norm on this space is defined by}
\begin{equation}
\|g\|_{\dot{H}^1} := \|\nabla g \|_{L^2}.
\end{equation}

Numerically solving \eqref{eqn:ms_ellip} is far more computationally expensive than solving the homogenized equation \eqref{eqn:homogenized}, motivating the wish to find the homogenized
coefficient $\overline{A}$ defining equation \eqref{eqn:homogenized}. The difficult part of obtaining the equation \eqref{eqn:homogenized} is solving the cell problem \eqref{eqn:cellprob}. Although explicit solutions exist in the one-dimensional setting for piecewise constant $A$ \cite{bhattacharya2023learning} and in the two-dimensional setting where $A$ is a layered material \cite{pavliotis2008multiscale}, in general a closed form solution is not available and the cell problem must be solved numerically. \mt{Note that in general the action of the divergence
$\nabla\cdot$ on terms involving $A$ in the cell problem necessitates the use of weak solutions  for $A\notin C^1(\Td, \Rdd);$ this is a commonly occurring situation in applications} such as those arising
from porous medium flow, or to vector-valued generalizations of the setting here to elasticity, rendering the numerical solution non-trivial. For this reason, it is potentially valuable to approximate the solution map 
\begin{equation}
\label{eq:themap}
    G: \; A \mapsto \chi,
\end{equation}
defined by the cell problem, using a map defined by a neural operator.
More generally it is foundational to the broader program of learning
homogenized constitutive models from data to thoroughly study this issue
for the divergence form elliptic equation as the insights gained will be important for
understanding the learning of more complex parameterized homogenized models, such as
those arising in nonlinear elasticity, viscoelasticity, and plasticity.

The full map from $A$ to the homogenized tensor $\overline{A}$ is expressed by $A \mapsto (\chi,A) \mapsto \overline{A}$, and one could instead learn the map 
\begin{equation}\label{eq:theothermap}
F: A \mapsto \overline{A}.
\end{equation} 
Since the map $(\chi,A) \mapsto \overline{A}$ is is defined by a quadrature, we focus on the approximation of $A \mapsto \chi$ and state equivalent results for the map $A \mapsto \overline{A}$ that emerge as consequences of the approximation of $\chi$. In this paper we
make the following contributions:
\begin{enumerate}
\item We state and prove universal approximation theorems for the \mt{map $G$ defined by \eqref{eqn:cellprob} and \eqref{eq:themap}, and map} $F$ defined by \eqref{eq:cell}, \eqref{eqn:cellprob}, and \eqref{eq:theothermap}.
\item We provide explicit examples of microstructures which satisfy the hypotheses
of our theorems; \mt{these include microstructures generated by
probability measures which generate discontinuous functions in $\BV$.}
\item We provide numerical experiments to demonstrate the ability of neural operators to approximate the solution map on four different classes of material parameters $A$, \mt{all covered by our theoretical setting.}
\end{enumerate}

In Subsection \ref{ssec:LR} we provide an overview of the literature, followed in Subsection \ref{ssec:SE} by
a discussion of stability estimates for \eqref{eqn:cellprob}, with respect to variations in $A$; these are
at the heart of the analysis of universal approximation.
The main body of the text then commences with Section \ref{sec:M}, which characterizes the
microstructures of interest to us in the context of continuum mechanics.
Section \ref{sec:ApproxThms} states universal approximation theorems for $G(\cdot)$ and $F(\cdot)$,
using the Fourier neural operator. 
In Section \ref{sec:E} we give numerical experiments illustrating the approximation of the map $G$ defined by \eqref{eq:themap}  on microstructures of interest in continuum mechanics. Details of the
stability estimates, the proofs of universal approximation theorems, properties of
the microstructures, and details of numerical experiments are given in Appendices \ref{apdx:perturb}, \ref{apdx:approxthms}, \ref{apdx: micro}, \mt{and \ref{apd:implementation} }respectively.

\subsection{Literature Review}
\label{ssec:LR}

Homogenization aims to derive macroscopic equations that describe the effective 
properties and behavior of solutions to problems at larger scales given a system that exhibits
behaviour at (possibly multiple) smaller scales. Although it is developed for the various cases of random,
statistically stationary, and periodic small-scale structures, we work here entirely in the periodic setting.
The underlying assumption of periodic homogenization theory is that the coefficient is periodic in the small-scale variable, and that the scale separation is large compared to the macroscopic scales of interest.
Convergence of the solution of the multiscale problem to
the homogenized solution is well-studied; see 
\cite{allaire1992homogenization,cioranescu1999introduction}. 
We refer to  the texts \cite{bensoussan2011asymptotic,blanc2023homogenization,pavliotis2008multiscale} 
for more comprehensive citations to the literature.
Homogenization has found extensive application in the 
setting of continuum mechanics \cite{gurtin1982introduction} where,
for many multiscale materials, the scale-separation assumption is
natural.
In this work, we are motivated in part by learning constitutive 
models for solid materials, where crystalline microstructure renders the material parameters
 discontinuous and may include corner interfaces. 
This difficulty has been explored extensively in the context of numerical methods for 
PDEs, particularly with adaptive finite element
 methods \cite{hou1997multiscale, bonito2013adaptive,nochetto2009theory, OWHADI2008397}.

There is a significant body of work on the approximation theory associated with
parametrically dependent solutions of PDEs, including viewing these solution as
a map between the function space of the parameter and the function space of the
solution, especially for problems possessing holomorphic regularity \cite{cohen2010convergence,cohen2011analytic,chkifa2013sparse}. This work could potentially
be used to study the cell problem for homogenization that is our focus here. However,
there has been recent interest in taking a data-driven approach to solving PDEs
 via machine learning because of its flexibility and ease of implementation.
 A particular approach to learning solutions to PDEs is operator learning, a machine learning 
 methodology where the map to be learned is viewed as an operator acting
 between infinite-dimensional function spaces rather than between finite-dimensional spaces \cite{bhattacharya2021model,li2021fourier,lu2021learning,nelsen2021random,kovachki2023neural}. 
  Determining whether, and then when, operator learning models have advantages over classical numerical methods
   in solving PDEs remains an active area of research \cite{batlle2023kernel}.
   The paper \cite{marcati2023exponential} makes a contribution to this area, in the
   context of the divergence form elliptic PDE and the map from coefficient to solution 
   when the coefficient is analytic over its domain; the authors prove that $\epsilon$ error is
   achievable for a DeepONet \cite{lu2021learning} of size only polylogarithmic in $\epsilon$,
   leveraging the exponential convergence of spectral collocation methods for boundary value problems with
   analytic solutions.
   However, in the setting of learning homogenized constitutive laws in material science, discontinuous
   coefficients form a natural focus and indeed form the focus of this paper. A few 
   characteristics make operator learning a promising option in this context. First, 
   machine learning has been groundbreaking in application settings with no 
   clear underlying equations, such as computer vision and language models 
   \cite{he2016deep, brown2020language}. In constitutive modeling, though 
   the microscale constitutive laws are known, the homogenized equations are 
   generally unknown and can incorporate dependencies that are not present on 
   the microscale, such as history dependence, anisotropy, and slip-stick 
   behavior \cite{phillips2001crystals, bhattacharya1999phase}.
   Thus, constitutive models lie in a partially equation-free 
setting where data-driven methods could be useful. Second, machine learned models
as surrogates for expensive computation can be valuable when the cost of producing 
data and training the model can be amortized over many forward uses of the trained model.
Since the same materials are often used for fabrication over long time periods, 
this can be a setting where the upfront cost of data production and model training
 is justified. 

Other work has already begun to explore the use of data-driven methods for 
constitutive modeling; a general review of the problem and its challenges, in the context of 
constitutive modeling of composite materials, may be found in \cite{liu2021review}. Several works use the popular framework of physics-informed machine learning to approach the problem \cite{fuhg2022physics, tartakovsky2020physics, ma2023learning, haghighat2023constitutive}. In \cite{as2022mechanics}, physical constraints are enforced on the network architecture while learning nonlinear elastic constitutive laws. In \cite{linka2021constitutive}, the model is given access to additional problem-specific physical knowledge. Similarly, the work of \cite{xu2021learning} predicts the Cholesky factor of the tangent stiffness matrix from which the stress may be calculated; this method enforces certain physical criteria. The paper
\cite{huang2020learning} studies approximation error and uncertainty quantification for this learning problem. In \cite{han2023neural}, a derivative-free approach is taken to learning homogenized solutions where regularity of the material coefficient is assumed. The work of \cite{liu2022learning} illustrates the potential of operator learning methodology to model constitutive laws with history dependence, such as those that arise in crystal plasticity. Finally, a number of further works demonstrate empirically the potential of learning constitutive models, including \cite{mozaffar2019deep, logarzo2021smart,deelia2021, liu2020neural}.

However, the underlying theory behind operator learning for constitutive models lags behind its empirical application. In \cite{bhattacharya2023learning}, approximation theories are developed to justify the use of a recurrent Markovian architecture that performs well in application settings with history dependence. This architecture is further explored in \cite{liu2023learning} with more complex microstructures. Universal approximation results are a first step in developing theory for learning because they guarantee that there exists an $\epsilon$-approximate operator
within the operator approximation class, which is consistent with an assumed true model underlying the data \cite{cybenko1989approximation,lanthaler2022error,kovachki2023neural,kovachki2021universal}. 
In addition to universal approximation, further insight may be gained by seeking to quantify the data or model size required to obtain a given level of accuracy; the papers
\cite{lanthaler2022error,kovachki2021universal, marcati2022exponential} also contain work in this direction, as do the papers
\cite{herrmann2022neural,opschoor2022exponential}, which build on the analysis developed in
\cite{cohen2010convergence,cohen2011analytic,chkifa2013sparse} referred to above. In our work
we leverage two existing universal approximation theorems for neural operators, one from \cite{kovachki2023neural} for general neural operators (NOs) and one from \cite{kovachki2021universal} for Fourier neural operators (FNOs), a particular practically useful
architecture from within the NO class. We take two different approaches to proving approximation theorems based on separate PDE solution stability results in pursuit of a more robust understanding of the learning problem. Since the state of the field is in its
infancy, it is valuable to have different approaches to these
analysis problems. Finally, we perform numerical experiments on various microstructures to understand the practical effects of non-smooth PDE coefficients in learning solutions.
We highlight the fact that in this paper we do not tackle issues related to the non-convex
optimization problem at the heart of training neural networks; we simply use state of the art stochastic
gradient descent for training, noting that theory explaining its excellent empirical behaviour is lacking.

Throughout this paper we focus on equation \eqref{eqn:ms_ellip}, which describes a conductivity
equation in a heterogeneous medium; a natural generalization of interest is to the constitutive law of linear elasticity, in which the solution is vector-valued and the coefficient is a fourth order tensor. Though it is a linear elliptic equation, we echo the sentiment of Blanc and Le Bris \cite{blanc2023homogenization} with their warning ``do not underestimate the difficulty of equation \eqref{eqn:ms_ellip}.'' There are many effects to be understood in this setting, and resolving learning challenges is a key step towards understanding similar questions for the learning of parametric
dependence in more complex homogenized constitutive laws
where machine-learning may prove particularly useful.

\subsection{Stability Estimates}
\label{ssec:SE}

At the heart of universal approximation theorems is stability of the solution map
\eqref{eq:themap}; in particular continuity of the map for certain classes of $A$. In this subsection, we present three key stability results that are used to prove the approximation theorems in Section \ref{sec:ApproxThms}. The proofs of the following stability estimates may all be found in Appendix \ref{apdx:perturb}.

A first strike at the stability of the solution map \eqref{eq:themap} is a modification of
the classic $L^{\infty}/H^1$ Lipschitz continuity result for dependence of the
solution of elliptic PDEs on the coefficient; here generalization is necessary because the coefficient
also appears on the right-hand side of the equation defining $G(\cdot):$

\begin{restatable}{proposition}{Propstabinfty}
    \label{prop:stab_infty}
    Consider the cell problem 
    defined by equation \eqref{eqn:cellprob}. The following hold:
    \begin{enumerate}
        \item If $A \in \PD$, then \eqref{eqn:cellprob} has a unique solution $\chi \in \dot{H}^1(\Td;\Rd)$ and 
        \[\|\chi \|_{\dot{H}^1(\Td;\Rd)} \leq \frac{\sqrt{d} \beta}{\alpha}.\]
        \item For $\chione$ and $\chitwo$ solutions to the cell problem in equation \eqref{eqn:cellprob} associated with coefficients $\Aone,\Atwo \in \PD$, respectively, it follows that
    \begin{equation}\label{eqn:perturb}
    \|\chitwo - \chione\|_{\dot{H}^1(\Td;\Rd)} \leq \frac{\sqrt{d}}{\alpha}\left(1 + \frac{\beta}{\alpha}\right)\|A^{(1)}-A^{(2)}\|_{L^{\infty}(\Td;\Rdd)}.
    \end{equation}
    \end{enumerate}
\end{restatable}

 However, this perturbation result is insufficient for approximation theory because the space $L^\infty$ is not separable \mt{and it is not
 natural to develop approximation theory in such spaces \cite[Chapter 9]{devore1993constructive}. While it is possible to} define the problem on a separable subspace of $L^\infty$, see Lemma~\ref{lemma:extended_map_Linfty}, such spaces are not particularly useful in applications to micromechanics. \mt{Many natural models for realistic microstructures work with classes of discontinuous functions
in which the boundary of material discontinuity can occur anywhere in the domain.
Such functions cannot be contained in any separable subspace of $L^\infty$; see Lemma~\ref{lemma:no_separable_characteristics}. To
deal with this issue it is desirable to establish} continuity from $L^q$ to $\dot{H}^1$ for some $q \in [2,\infty)$. To this end, we provide two additional stability results. The first stability result gives continuity, but not Lipschitz continuity, from $L^2$ to $\dot{H}^1$. The second stability result gives Lipschitz continuity from $L^q$ to $\dot{H}^1$, some $q \in (2,\infty)$.

\begin{restatable}{proposition}{PropstabLtwo}
        \label{prop:cellproblem_cont_L2}
Endow \(\PD\) with the \(L^2(\Td;\Rdd)\) induced topology and let \(K \subset \PD\) be a closed set. Define the mapping \(G: K \to \dot{H}^1(\Td;\Rd)\) by \(A \mapsto \chi\) as given by \eqref{eqn:cellprob}. Then there exists a bounded continuous mapping $$\G \in C(L^2(\Td; \;\Rdd); \dot{H}^1(\Td;\Rd))$$ such that \(\G(A) = G(A)\) for any \(A \in K\).
\end{restatable}

The preceding $L^2$ continuity proposition is used to prove the approximation results for the FNO in Theorems \ref{thm:fno_cell} and \ref{thm:fno_hom}. \mt{While not necessary for the approximation theory proofs, the following proposition on Lipschitz continuity from $L^q$ to $\dot{H}^1$ establishes a more concrete bound on the approximation error, which allows for additional analysis such as providing rough bounds on grid error as discussed in Subsection \ref{ssec:discussion}.}

\begin{restatable}{proposition}{PropstabLq}
\label{prop:stab_Lq}
    There exists $q_0 \in (2,\infty)$ such that, for all $q$ satisfying $q \in (q_0,\infty]$, the following holds.
    \mt{Endow $\PD$ with the $L^q(\Td; \Rdd)$ topology and let $K \subset \PD$ be a closed set. Define the mapping $G: K \to \dot{H}^1(\Td;\Rd)$ by $A \mapsto \chi$ as given by \eqref{eqn:cellprob}. Then there exists a bounded Lipschitz-continuous mapping $$\G: \; L^q(\Td;\Rdd)\to \dot{H}^1(\Td;\Rd)$$ such that $\G(A) = G(A)$ for any $A \in K$.}
\end{restatable}

\begin{remark}
    Explicit upper bounds for $q_0$ in Proposition \ref{prop:stab_Lq} exist and are discussed in Remark \ref{rem:q0}. 
\end{remark}

\section{Microstructures}
\label{sec:M}
The main application area of this work is constitutive modeling. In this section we describe various classes of microstructures that our theory covers. In particular, we describe four classes of microstructures in two dimensions: 
\begin{enumerate}
    \item Smooth microstructures generated via truncated, rescaled log-normal random fields. 
    \item Discontinuous microstructures with smooth interfaces generated by Lipschitz star-shaped inclusions.
    \item Discontinuous microstructures with square inclusions.
    \item Voronoi crystal microstructures.
\end{enumerate}
Visualizations of examples of these microstructures may be found in 
Figure \ref{fig:microstructures}. \mt{We emphasize that all four examples
lead to functions in $\BV$, a fact that we exploit in Section \ref{sec:E} when showing that
our abstract analysis from Section \ref{sec:ApproxThms} applies to them all.}

\begin{figure}[h!]
    \centering
    \includegraphics[width = 0.65\textwidth]{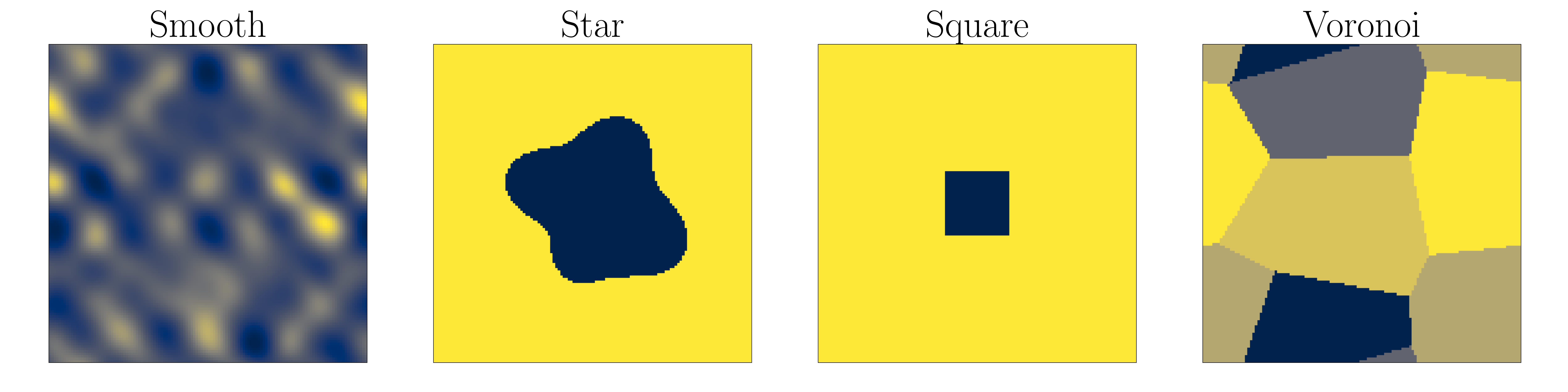}
    \caption{Microstructure Examples}
    \label{fig:microstructures}
\end{figure}
\FloatBarrier

\paragraph{Smooth Microstructures}
The smooth microstructures are generated by exponentiating a rescaled Gaussian random field.
$A$ is symmetric and coercive everywhere in the domain with a bounded eigenvalue ratio. Furthermore, the smooth function $A$ and its derivatives are Lipschitz. Our theory is developed specifically to analyze non-smooth microstructures, so this example is used mainly as a point of comparison.

\paragraph{Star Inclusions}
For the star inclusion microstructure, $A$ is taken to be constant inside and outside the star-shaped boundary. 
The boundary function is smooth and Lipschitz in each of its derivatives. $A$ is positive and coercive in both regions with a bounded eigenvalue ratio. This microstructure introduces discontinuities, but the boundary remains smooth.

\paragraph{Square Inclusions}
\mt{For the square inclusion microstructure, $A$ is taken to be constant inside and outside the square boundary.} Since we assume periodicity, without loss of generality the square inclusion is centered. The size of the square inclusion within the cell is varied between samples as are the constant values of $A$. This microstructure builds on the complexity of the star inclusion microstructure by adding corners to the inclusion boundary. 

\paragraph{Voronoi Interfaces}
The Voronoi crystal microstructures are generated by assuming a random Voronoi tessellation and letting $A$ be piecewise-constant taking a single value on each Voronoi cell. The values of $A$ on the cells and locations of the cell centers may be varied. This is the most complex microstructure among our examples and is a primary motivation for this work as Voronoi tessellations are a common model for crystal structure in materials. 

\section{Universal Approximation Results}
\label{sec:ApproxThms}

\mt{In this section we state the two approximation theorems for learning solution operators to the cell problem. Theorem \ref{thm:fno_cell} concerns learning the map $A \to \chi$ in equation \eqref{eqn:cellprob}, and Theorem \ref{thm:fno_hom} concerns learning the map $A \to \overline{A}$ described by the combination of equations \eqref{eqn:cellprob} and \eqref{eq:cell}. Theorems \ref{thm:fno_cell} and \ref{thm:fno_hom} are specific to learning a Fourier neural operator (FNO), which is a subclass of the general neural operator. The proofs of the theorems in this section may be found in Appendix \ref{apdx:approxthms}.}

\subsection{Definitions of Neural Operators}

First, we define a general neural operator (NO) and the Fourier neural operator (FNO). The definitions are largely taken from \cite{kovachki2023neural}, and we refer to this work for a more in-depth understanding of these operators. In this work, we restrict the domain to the torus. 

\begin{definition}[\bf{General Neural Operator}]  Let $\mathcal{A}$ and $\mathcal{U}$ be two Banach spaces of real vector-valued functions
over domain $\Td$. Assume input functions $a \in \mathcal{A}$ are $\R^{d_a}$-valued while the output functions $u \in \mathcal{U}$ are $\R^{d_u}$-valued. The neural operator architecture $\mathcal{G}_{\theta}:\mathcal{A} \to \mathcal{U}$ is
\begin{align*}
    \mathcal{G}_{\theta} &= \mathcal{Q} \circ \mathsf{L}_{T-1}\circ \dots \circ \mathsf{L}_0\circ \mathcal{P},\\
    v_{t+1} &= \mathsf{L}_t v_t  =\sigma_t(W_{t}v_t + \mathcal{K}_{t}v_t + b_{t}), \quad t=0,1, \dots, T-1
\end{align*}
with \(v_0 = \mathcal{P}(a)\), $u=\mathcal{Q}(v_T)$ and \(\mathcal{G}_\theta (a) = u\). Here, $\mathcal{P}: \R^{d_a} \to \R^{d_{v_0}}$ is a local lifting map, $\mathcal{Q}:\R^{d_{v_T}}\to\R^{d_u}$ is a local projection map
and the $\sigma_t$ are fixed nonlinear activation functions acting locally as maps $\R^{d_{v_{t+1}}} \to \R^{d_{v_{t+1}}}$ in each layer (with all of $\mathcal{P}$, $\mathcal{Q}$ and the $\sigma_t$ viewed
as operators acting pointwise, or pointwise almost everywhere, over the domain $\Td$), $W_t \in \R^{d_{v_{t+1}}\times d_{v_t}}$ are matrices, $\mathcal{K}_t: \{v_t: \Td \to \R^{d_{v_t}}\} \to \{v_{t+1}:\Td \to \R^{d_{v_{t+1}}}\}$ are integral kernel operators and $b_t: \Td \to \R^{d_{v_{t+1}}}$ are bias functions. For any $m \in \mathbb{N}_0$, the activation functions $\sigma_t$ are restricted to the set of continuous $\R \to \R$ maps which make real-valued, feed-forward neural networks dense in $C^m(\Rd)$ on compact sets for any fixed network depth.\footnote{We note that all globally Lipschitz, non-polynomial, $C^m(\R)$ functions belong to this class.} The integral kernel operators $\mathcal{K}_t$ are defined as
\[(\mathcal{K}_t v_t)(x) = \int_{\Td} \kappa_t (x,y) v_t(y) \: dy \]
with standard multi-layered perceptrons (MLP) \(\kappa_t : \Td \times \Td \to \R^{d_{v_{t+1}} \times d_{v_t}}\). We denote by $\theta$ the collection of parameters that specify $\mathcal{G}_{\theta}$, which include the weights $W_t$, biases $b_t$, parameters of the kernels $\kappa_t$, and the parameters describing the lifting and projection maps $\mathcal{P}$ and $\mathcal{Q}$ (usually also MLPs).
\label{def:GNO}
\end{definition}

The FNO is a subclass of the NO. 
\begin{definition}[\bf{Fourier Neural Operator}] The FNO inherits the structure and definition of the NO in Definition \ref{def:GNO}, together with some specific design choices. We fix $d_{v_t} = d_v$ for all $t$, where $d_v$ is referred to as the number of channels, or model width, of the FNO. We fix $\sigma_t=\sigma$ to be a globally Lipschitz, non-polynomial, $C^{\infty}$ function.\footnote{In this work in all numerical experiments we use the GeLU activation function as in \cite{li2021fourier}.} Finally, the kernel operators $\mathcal{K}_t$ are parameterized in the Fourier domain in the following manner. Let 
\begin{equation*}
    \psi_k(x) = \text{e}^{2\pi i \langle k, x \rangle}, \quad x \in \Td, \; k \in \mathbb{Z}^d,
\end{equation*}
denote the Fourier basis for $L^2(\Td;\mathbb{C})$ where \(i = \sqrt{-1}\) is the imaginary unit.
Then, for each $t$, the kernel operator $\mathcal{K}_t$ is parameterized by 
\begin{equation*}
    (\mathcal{K}_t v_t)_l (x) = \sum_{\substack{k \in \mathbb{Z}^d \\ |k| \leq k_{\max}}} \left(\sum_{j = 1}^{d_v}P_{lj}^k \langle (v_t)_j, \overline{\psi}_k\rangle_{L^2(\Td;\mathbb{C})}\right)\psi_k(x).
\end{equation*}
Here, \(l=1,\dots,d_{v}\) and each $P^k \in \mathbb{C}^{{d_{v}} \times d_{v}}$ constitute the learnable parameters of the integral operator.
\end{definition}


From the definition of the FNO, we note that parameterizing the kernels in the Fourier domain allows for efficient computation using the FFT. We refer to \cite{kovachki2023neural,li2021fourier} for additional details. 

Finally we observe that in numerous applications, an example being
learning of the map  \(A \mapsto \bar{A}\) \eqref{eq:cell}, \eqref{eqn:cellprob}, it is desirable to modify the FNO so that the output space is simply a Euclidean space, and not a function space;
this generalization is explored in \cite{nickdanielmargaret2023}.
An alternative approach, exemplified by Theorem \ref{thm:fno_hom} in the
next subsection, is to allow the FNO output to be a function that may be evaluated at any point in the domain to yield an approximation of the point in Euclidean space.

\subsection{Main Theorems}

These two theorems guarantee the existence of an FNO approximating the maps $A \mapsto \chi$ and $A \mapsto \overline{A}$ and are based on the stability estimate for continuity from $L^2 \to \dot{H}^1$ obtained in Proposition \ref{prop:cellproblem_cont_L2}. Both theorems are proved in Appendix \ref{apdx:approxthms}.
\begin{restatable}{theorem}{thmLtwochi}

\label{thm:fno_cell}
Let \(K \subset \PD \) and define the mapping \(G : K \to \dot{H}^1 (\Td;\Rd)\) by \(A \mapsto \chi\) as given by \eqref{eqn:cellprob}. \mt{Assume in addition that $K$ is compact in $L^2(\Td;\Rdd)$. Then, for any \(\epsilon > 0\), }there exists an FNO \(\Psi : K \to \dot{H}^1(\Td;\Rd)\) such that
\[\sup_{A \in K} \|G(A) - \Psi(A)\|_{\dot{H}^1} < \epsilon.\]
\end{restatable}

\begin{restatable}{theorem}{thmLtwoAbar}
    \label{thm:fno_hom}
Let \(K \subset \PD\) and define the mapping \(F : K \to \Rdd\) by \(A \mapsto \bar{A}\) as given by \eqref{eq:cell}, \eqref{eqn:cellprob}. \mt{Assume in addition that $K$ is compact in $L^2(\Td;\Rdd)$.} Then, for any \(\epsilon > 0\), there exists an FNO \(\Phi: K \to L^{\infty}(\Td;\Rdd)\) such that
\[\mt{\sup_{A \in K} \sup_{x \in \Td}|F(A) - \Phi(A)(x)|_F < \epsilon.}\]
\end{restatable}

The above approximation results can also be formulated to hold, on average, over any probability measure with a finite second moment that is supported on $\PD$. In particular, if we let $\mu$ be such a probability measure then there exists an FNO or a neural operator \(\Psi\) such that
\begin{equation}
\label{eq:QQQ}
\mathbb{E}_{A \sim \mu} \|G(A) - \Psi(A)\|_{\dot{H}^1} < \epsilon.
\end{equation}
This follows by applying Theorem 18 from \cite{kovachki2021universal} in the respective proofs instead of Theorem 5 from the same work.  We do not carry out the full details here. While this allows approximation over the non-compact set $\PD$, the error can only be controlled on average instead of uniformly. \mt{In Section \ref{sec:E}, inputs are generated via  probability measures supported on compact subsets of $L^2$; 
thus both the approximation Theorem \ref{thm:fno_cell}, and 
its analog in the form \eqref{eq:QQQ}, are relevant.}

\section{Numerical Experiments}
\label{sec:E}
In this section, we show that it is possible to find good operator approximations of the homogenization map \eqref{eq:themap}, defined
by \eqref{eqn:cellprob}, in practice.
We focus on use of the FNO and note that, while
Theorems \ref{thm:fno_cell} and \ref{thm:fno_hom} assert the
existence of desirable operator approximations, they are not constructive and do not come equipped with error estimates.
\mt{We find approximations using standard empirical loss minimization
techniques and, by means of numerical experiments, quantify the complexity with respect to volume of
data and with respect to size of parametric approximation.}

\mt{We work with the microstructures from Section \ref{sec:M}. In
this context we note that Theorems \ref{thm:fno_cell} and \ref{thm:fno_hom} apply. To demonstrate this it is necessary
to establish that the subsets of coefficient functions employed are compact in $L^2$. We achieve this by noting that all our sets of coefficient functions are contained in $\PD \cap \BV$, as defined in Subsection \ref{ssec:PF}. Then we use Lemma \ref{lem:BVLinfty_L2compact} to establish compactness of these subsets of coefficient functions in $L^2$.} The smooth microstructure example serves as a comparison case
for examining the impact of discontinuous coefficients on the learning accuracy.  \mt{The remaining three examples present different 
approximation theoretic challenges including curved boundaries
(star inclusions), corners (square inclusions) and junctions of
several domains (Voronoi).}

The experiments are all conducted using an FNO with a fixed number
$T=4$ of hidden layers. The two remaining parameters to vary are the
channel width $d_v$ and the number of Fourier modes $k_{max}$. \mt{For implementation details, see Appendix \ref{apd:implementation}.}
We make the following observations based on the numerical experiments.
\begin{enumerate}
\item The effective $\overline{A}$ tensors computed from the model predicted solutions exhibit \mt{relative error under $1\%$} for all examples; \mt{the effective  $\overline{A}$ is computed from the learned cell problem solution $\chi$ using equation \eqref{eq:cell}.}
    \item The error in the learned $\chi$ is significantly higher along discontinuous material boundaries and corner interfaces, as expected. However, the FNO operator approximation is able to approximate the solution with reasonable relative error even for the most complex case; this most complex case concerns the set of input functions with varying Voronoi geometry and varying microstructural properties within the domain.
    \item In comparison with the smooth microstructure case, learning the map for the Voronoi microstructure requires substantially
    more data to avoid training a model which plateaus at a poor
    level of accuracy.
    \item When compared with the smooth microstructure case, the error for the \linebreak Voronoi microstructure decreases more slowly with respect to increasing \linebreak model width, but shows more favourable response with
    respect to increasing the number of Fourier modes.
    \item \mt{Models trained at one discretization may be evaluated at different discretizations for both the smooth and Voronoi microstructures as is characteristic of the FNO. The Voronoi microstructure exhibits, empirically, greater
    robustness to changes in discretization.} 
    \end{enumerate}
    
We first describe implementation details of each of the microstructures in Subsection \ref{ssec:micro_imple}. Then we show outcomes of the numerical experiments in Subsection \ref{ssec:results},  
discussing them in Subsection \ref{ssec:discussion}.

\subsection{Microstructure Implementation}\label{ssec:micro_imple}
\vspace{0.1cm}

For each microstructure, two positive eigenvalues and three components of the two eigenvectors are randomly generated, and the final eigenvector component is chosen to enforce symmetry. All eigenvalue ratios are at most $e^2$ by construction. In this manner, $A$ is symmetric and coercive and has a bounded eigenvalue ratio.

\paragraph{Smooth Microstructures}
The smooth microstructures are generated by exponentiating a rescaled approximation of a Gaussian random field. 
\mt{The random field used to generate the eigenvalues and three eigenvector components of $A(x)$} is
as follows:
\begin{align*}\label{eqn:smooth_eqn}
    \widehat{\lambda}_i(x) &= \sum_{k_1,k_2 = 1}^4 \xi^{(1)}_{k_1,k_2} \sin\left(2\pi k_1 x_1\right)\cos(2\pi k_2 x_2) + \xi^{(2)}_{k_1,k_2}\cos\left(2\pi k_1 x_1\right)\sin\left(2\pi k_2 x_2\right),\\
    \lambda_i(x) & = \exp\left(\frac{\widehat{\lambda}_i(x)}{\max_{x' \in [0,1]^2} |\widehat{\lambda}_i(x')|}\right),
\end{align*}
where $\xi^{(j)}_{k_1,k_2}$ are i.i.d. normal Gaussian random variables. 

\paragraph{Star-Shaped Inclusions}
The star-shaped inclusions are generated by defining a random Lipschitz polar boundary function as 
\begin{align*}
    r(\theta) & = \mathsf{a} + \mathsf{b} \sum_{k = 1}^5 \xi_k \sin(k \theta) 
\end{align*}
where $\xi_k$ are \mt{i.i.d. uniform random variables $U[-1,1]$, and $\mathsf{a}$ and $\mathsf{b}$ are constants that guarantee $0 < \epsilon < r < 0.5-\epsilon$ for some fixed $\epsilon >0$.}
Then $A(x)$ is constant inside and outside the boundary. We randomly sample eigenvalues for $A$ on each domain 
\mt{via $\lambda_i \sim U[e^{-1},e]$}. The three components of 
the eigenvectors are i.i.d. normal random variables.

\paragraph{Square Inclusions}
The radius of the square is randomly generated via 
\begin{equation*}
    r = \mathsf{a} + \mathsf{b} \zeta
\end{equation*}
where $\zeta$ is a uniform random variable on $[0,1]$ and \mt{$\mathsf{a}$ and $\mathsf{b}$ are positive constants that guarantee $0 < \epsilon < r < 0.5-\epsilon$ for some fixed $\epsilon >0$.} The values of $A$ on 
each of the constant domains are chosen in the same manner as in the star-shaped inclusion case.

\paragraph{Voronoi Interfaces}
The Voronoi crystal microstructure has constant $A$ on each Voronoi cell and is chosen uniformly at random in the same manner as for the star inclusions. Voronoi tessellations are a common model for crystal structure in materials. In one Voronoi example,  we fix the geometry for all data, and in a second Voronoi example we vary the geometry by randomly sampling five cell centers from a uniform distribution on the unit square. 

\subsection{Results}\label{ssec:results}
Each FNO model is trained using the empirical estimate of the mean squared $H^1$ norm: 
\begin{equation}
    \text{Loss}(\theta) = \frac{1}{N}\sum_{n = 1}^N \Bigl(\|\chi^{(n)} - \widehat{\chi}^{(n)}\|^2_{L^2} + \|\nabla \chi^{(n)} - \nabla \widehat{\chi}^{(n)}\|^2_{L^2}\Bigr)
\end{equation}
where $n$ is the sample index, $\chi$ is the true solution, and $\widehat{\chi}$ is the FNO approximation of the solution parameterized by $\theta.$ In the analysis, we examine several different measures of error, including the following relative $H^1$ and relative $W^{1,10}$ errors. 
\begin{subequations}\label{eqn:errors}
    \begin{equation}
    \text{Relative $H^1$ Error (RHE)} = \frac{1}{N}\sum_{n=1}^N \left(\frac{\|\chi^{(n)}- \widehat{\chi}^{(n)}\|^2_{L^2} + \|\nabla \chi^{(n)} - \nabla \widehat{\chi}^{(n)}\|^2_{L^2}}{\|\chi^{(n)}\|^2_{L^2} + \|\nabla\chi^{(n)}\|^2_{L^2}}\right)^{\frac{1}{2}} 
    \end{equation}
    \begin{equation}
    \text{Relative $W^{1,10}$ Error (RWE)} = \frac{1}{N}\sum_{n=1}^N \left(\frac{\|\chi^{(n)}- \widehat{\chi}^{(n)}\|^{10}_{L^{10}} + \|\nabla \chi^{(n)} - \nabla \widehat{\chi}^{(n)}\|^{10}_{L^{10}}}{\|\chi^{(n)}\|^{10}_{L^{10}} + \|\nabla\chi^{(n)}\|^{10}_{L^{10}}}\right)^{\frac{1}{10}} 
    \end{equation}
\end{subequations}

\mt{The $W^{1,10}$ norm gives a sense of the higher errors that occur at interfaces, corners and functions. We could have used $W^{1,p}$ 
for any $p$ large enough.}


Finally, we also look at error in $\overline{A}$, which we scale by the difference between the arithmetic and harmonic mean of $A$. Any effective $\overline{A}$ should have a norm in this range; \mt{these are known in mechanics as Voigt-Reuss bounds and have a physical interpretation as bounds obtained via energy principles by ignoring equilibrium for the upper bound (arithmetic mean) and ignoring compatibility for the lower bound (harmonic mean) \cite{hill1952elastic}.} The resulting error measure is given by 
\begin{equation}\label{eqn:err_A}
    \text{Relative $\overline{A}$ Error (RAE) } = \frac{\|\overline{A} - \widehat{\overline{A}}\|_F}{a_m - a_h} 
\end{equation}
where the arithmetic mean $a_m$ and harmonic mean $a_h$ are given by 
\begin{align*}
    a_m & = \left\|\int_{\mathbb{T}^2} A(x) \dx \right\|_F\\
    a_h & = \left\|\left(\int_{\mathbb{T}^2} A^{-1}(x)\dx \right)^{-1}\right\|_F.
\end{align*}
We note that using $a_m-a_h$ rather than $\|\overline{A}\|_F$ as a scaling factor in equation \eqref{eqn:err_A} leads to a larger error value, so achieving low error in this measure of distance is harder.

We train models on five different datasets. Visualizations of the median-error test samples for each example may be viewed in Figure \ref{fig:results_vis}, and the numerical errors are shown in Figure \ref{tab:Errors}. Each of these models is trained on $9500$ data samples generated using an FE solver on a
triangular mesh with the solution interpolated to a $128 \times 128$ grid. \mt{Additional model details may be found in Appendix \ref{apd:implementation}.} 

\mt{We perform an experiment to test the discretization-robustness of the FNO model, results of which are shown in Figure \ref{fig:disc-err}. The models are trained with data from the resolution $128 \times 128$ and evaluated on test data with different resolution. We emphasize that evaluating the FNO on different resolution is trivial in implementation by design.}

\mt{We also investigate the effects of the number of training data and the model size on the error for the smooth and Voronoi microstructures; 
similar experiments, for different operator learning problems,
are presented in \cite{de2022cost}.} A plot of error versus training data may be found in Figure \ref{fig:data-size-compare}, and plots of error versus the number of Fourier modes for fixed total model size, as measured by (model width) $\times$ (number Fourier modes), may be found in Figure \ref{fig:err-vs-size}.
Figure \ref{fig:err-vs-size} addresses the question of how to
optimally distribute computational budget through different parameterizations to achieve minimum error at given cost as measured
by number of parameters; it should be compared to similar experiments
in \cite{lanthaler2023nonlocal}.

\begin{figure}\label{fig:results_vis}
\centering
\includegraphics[width = 1.0\textwidth]{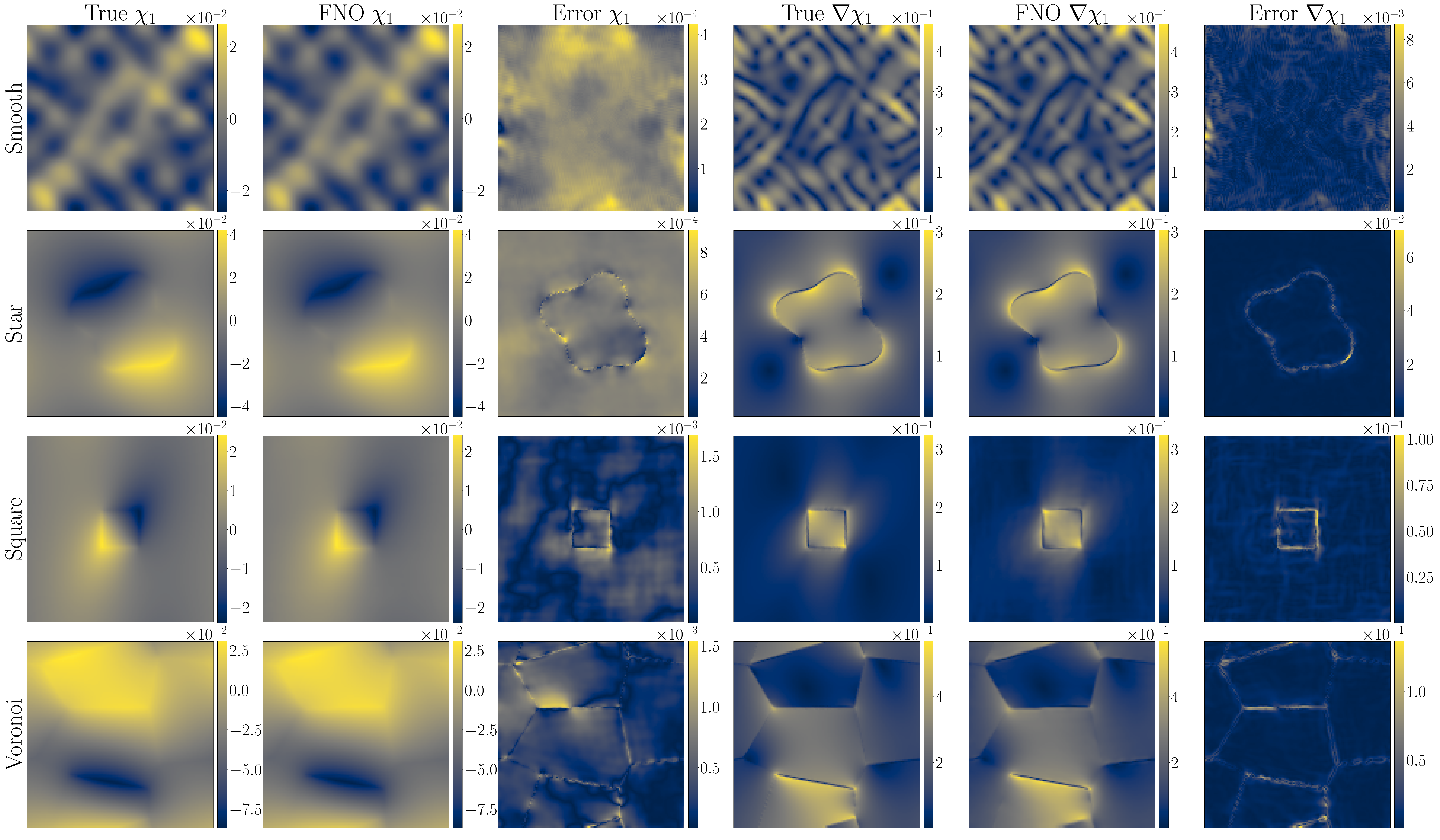}
\caption{\mt{Visualization of the trained models evaluated on test samples that gave median relative $H^1$ error for each microstructure. The microstructure inputs of each row correspond to those of Figure \ref{fig:microstructures}. The first shows the true $\chi_1$, the second shows the $FNO$ predicted $\chi_1$, and the third shows the absolute value of the error between the true and predicted $\chi_1$. The fourth column shows the 2-norm of the gradient of the true $\chi_1$, and the fifth shows the 2-norm of the gradient of the predicted $\chi_1$. The last column shows the 2-norm of the difference between the two gradients.}}
\end{figure}

\begin{figure}
\vspace{0.1cm}
\begin{minipage}{0.63\textwidth}
\begin{adjustbox}{width=\textwidth}
\resizebox{\textwidth}{!}{%
\begin{tabular}{|| c | c | c |  c ||}
 \hline
 Microstructure & Mean RHE  & Mean RWE & Median RAE\\  
 \hline\hline
 Smooth  &   $0.0062 \pm 1\cdot10^{-4}$ &  $0.0091\pm 1\cdot 10^{-4}$ &  $0.0007\pm 1\cdot 10^{-5}$ \\ 
 \hline
 Star  &  $0.0313 \pm 1\cdot 10^{-4}$ & $0.1318\pm 5\cdot 10^{-4}$&  $0.0014 \pm 3\cdot 10^{-5}$\\
 \hline
 Square  &  $0.1012 \pm 5 \cdot 10^{-4}$ &  $0.2741\pm 2\cdot 10^{-3}$&  $0.0047 \pm 1\cdot 10^{-4}$\\
 \hline
Voronoi &  $0.0565 \pm 4\cdot 10^{-4}$&  $0.2129\pm 3\cdot 10^{-3}$&  $0.0027 \pm 8\cdot 10^{-5}$ \\
 \hline
 \makecell{Voronoi \\ (Fixed Geometry)}  &  $0.0073\pm 3\cdot10^{-5}$ &  $0.0140\pm 3\cdot10^{-4}$ &  $0.0007 \pm 2\cdot 10^{-5}$\\
 \hline
\end{tabular}}
\end{adjustbox}
 \caption{\mt{Errors for each each numerical experiment; five sample models are trained for each microstructure. The expressions for the RHE (Relative $H^1$ Error), RWE (Relative $W^{1,10}$ Error) and RAE (Relative $\overline{A}$ Error) may be found in equations \eqref{eqn:errors} and \eqref{eqn:err_A}. The errors are evaluated over a test set of size $500$. All examples have varying geometry except the second Voronoi example.}}
 \label{tab:Errors}
 \end{minipage}\hfill
 \begin{minipage}{0.33\textwidth}
 \vspace{-0.7cm}
    \includegraphics[width = 0.9\textwidth]{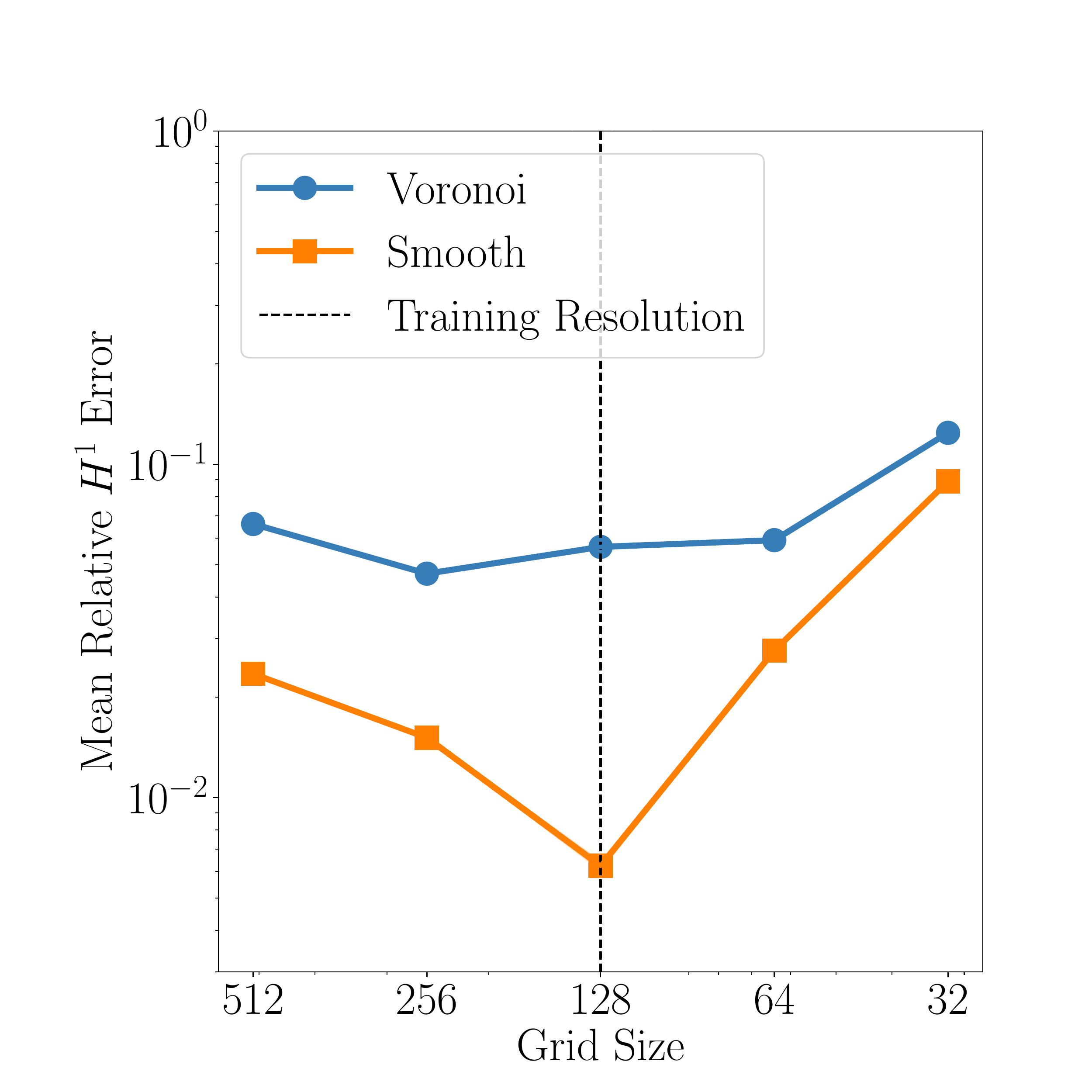}\vspace{-0.5cm}
    \caption{\mt{Five sample models trained on Smooth and Voronoi data at $128 \times 128$ grid resolution evaluated at different resolutions.}}
    \label{fig:disc-err}
  \end{minipage}
\end{figure}
\begin{figure}
  \begin{minipage}[t]{0.32\textwidth}
    \includegraphics[height = 3.8cm]{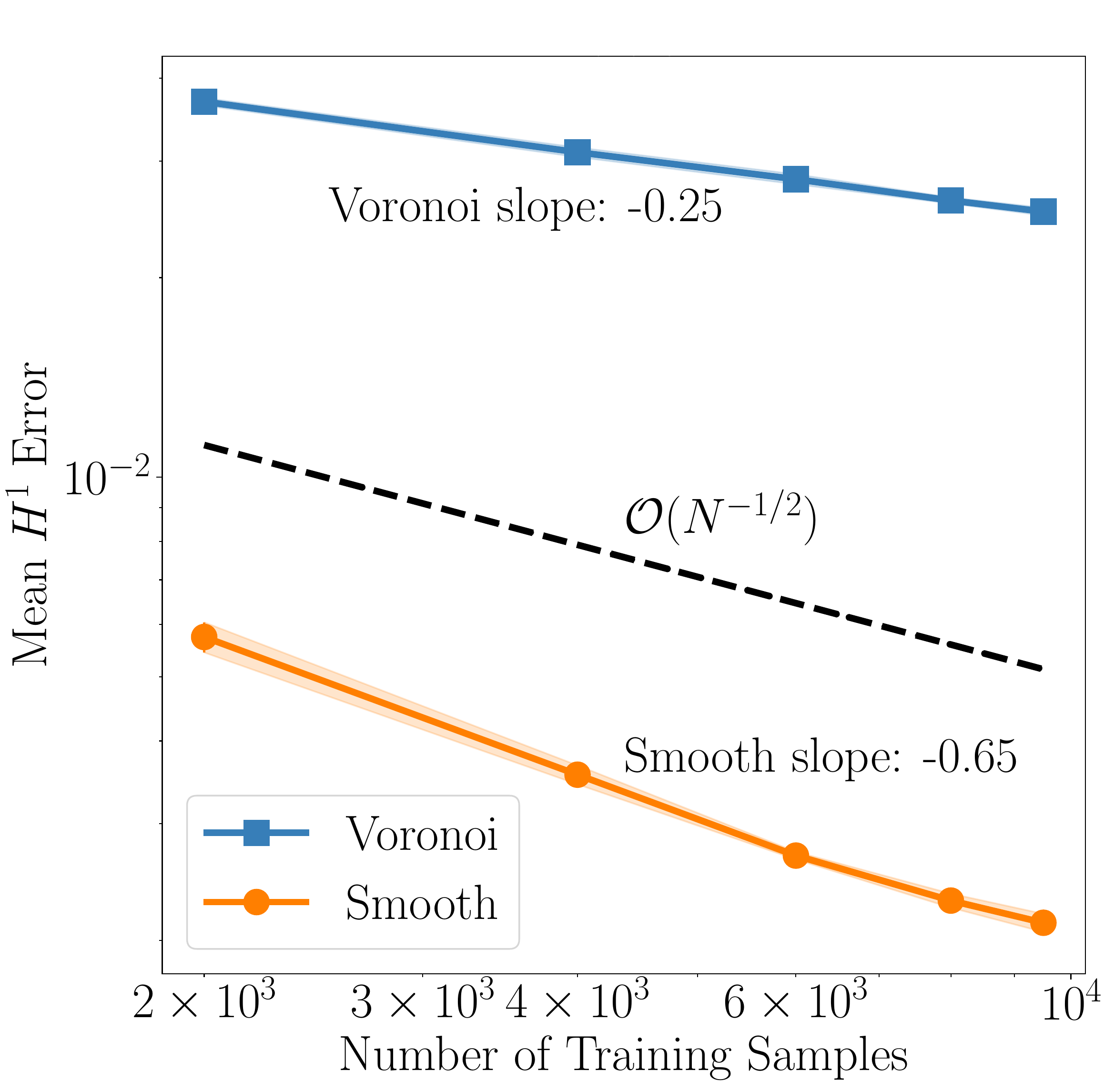}
    \caption{\mt{A comparison of test error for different amounts of training data for models trained on Voronoi and Smooth data. Five sample models are used for each data point.}}
    \label{fig:data-size-compare}
  \end{minipage}\hfill
  \begin{minipage}[t]{0.65\textwidth}
    \includegraphics[height = 3.96cm]{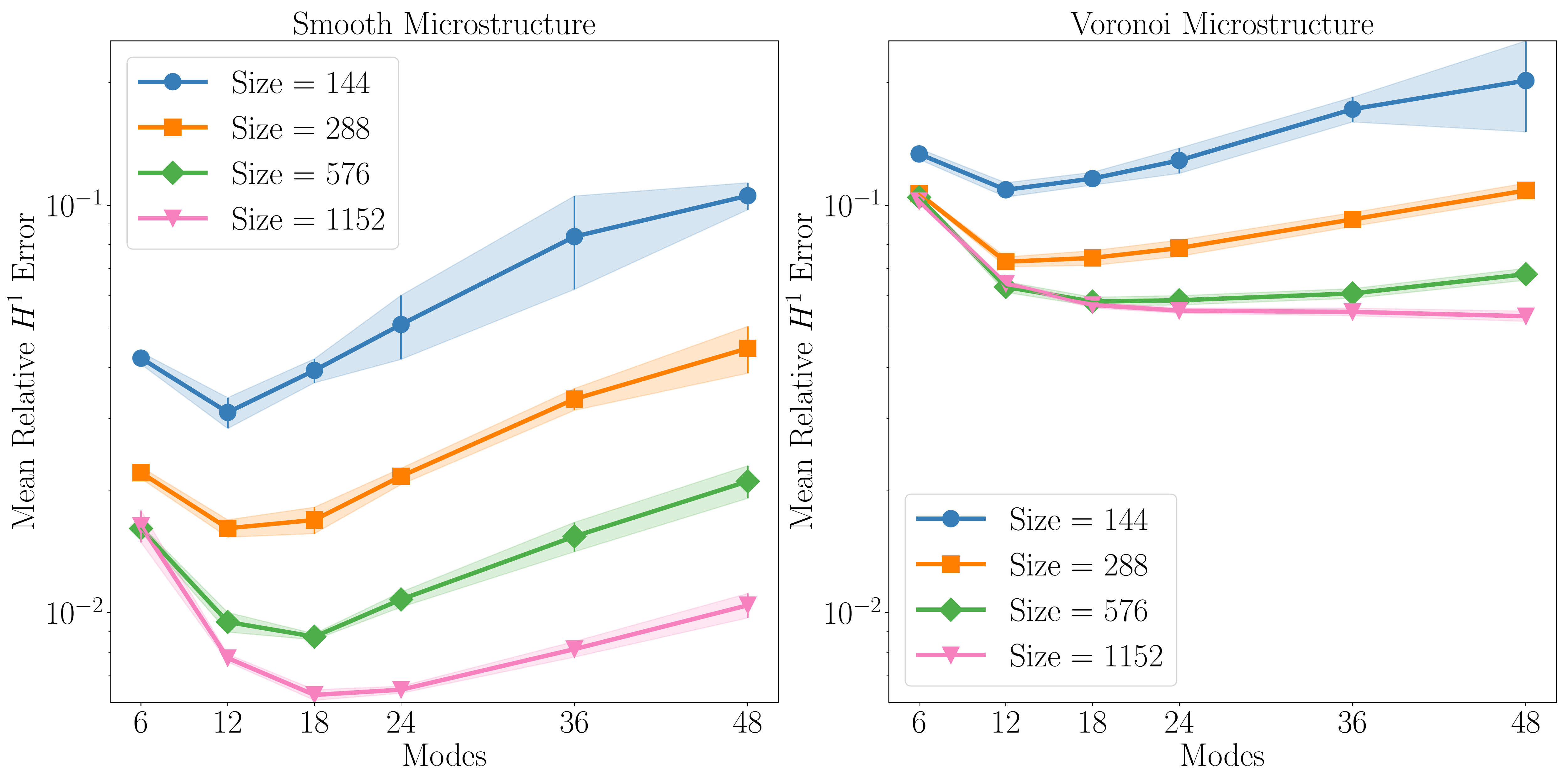}
    \caption{Relative $H^1$ error versus model size for the smooth and Voronoi examples with varying geometry. The number of Fourier modes in each direction and the model width were varied. Each line indicates a constant product of modes$\times$width. Training data size was fixed at 9500 samples, and five samples were used for each data point.}
    \label{fig:err-vs-size}
  \end{minipage}
\end{figure}

\subsection{Discussion}\label{ssec:discussion}

As can be seen from the data in Figure \ref{tab:Errors}, the microstructures exhibiting discontinuities lead to higher model error than the smooth microstructure, \mt{and the introduction of corner interfaces leads to further increase in error.} The visualizations of the median-error test samples in Figure \ref{fig:results_vis} give some intuition; error is an order of magnitude higher along discontinuous boundaries; this is most apparent in the gradient. The true solution gradient often takes its most extreme values along the discontinuities, and the RWE gives an indication of how well the model captures the most extreme values in the solution. Unsurprisingly, this error is much higher than the RHE, but we note that it is confined to a small area of the domain along discontinuous boundaries and corner interfaces. 

\mt{In the discretization-robustness experiment described in Figure \ref{fig:disc-err}, we observe that the Voronoi model exhibits greater robustness to changes in discretization. We hypothesize that, in the  \mt{direction of decreasing resolution}, the smaller error increase for the Voronoi model, in comparison with the smooth model, could be due to the piecewise-constant nature of the Voronoi microstructure on faces; improved resolution here does not help. On the other hand, for larger grid sizes, increased resolution on corners and discontinuities can help, which could explain the decrease in error from grid edge size of $128$ to $256$ for the Voronoi model while the smooth model increases in error. One could fine-tune the trained models with small amounts of data from different resolutions, but we leave this transfer learning exploration to future work.}

We also examine the effect of the number of training data samples and the FNO size on model accuracy for the smooth and Voronoi microstructures. \mt{For data size dependence, we observe in Figure \ref{fig:data-size-compare} that for these two microstructures, the test error scales $\approx N^{-0.65}$ and $\approx N^{-0.25}$, respectively, where $N$ is the number of training data. In theory, we do not expect to beat the Monte Carlo error decay of $\frac{1}{\sqrt{N}}$ \cite{murphy2012machine}.} We note that this is comparable to the behavior during training over $400$ epochs; the test error for the smooth microstructure continues to decrease over the entire training periodic, but the test error for the Voronoi microstructure plateaus by around $100$ epochs. The model size also presents a qualitatively different effect on error for the smooth and Voronoi microstructures. In Figure \ref{fig:err-vs-size}, we see the tradeoff between the number of Fourier modes and the model width for approximately constant model size, measured as the product of the width and number of modes. The Voronoi example benefits from additional Fourier modes, whereas the smooth example worsens. On the other hand, the smooth model benefits more from an increase in model width. We refer to \cite{de2022cost,lanthaler2023nonlocal} for in-depth numerical studies of errors, choice of hyperparameters,
and parameter distributions for FNO; here we highlight only the qualitative differences between the model behavior for different microstructures.

\mt{We also note that a significant portion of the model error may be attributed to grid ambiguity; with a $128 \times 128$ grid, the FNO does not know where between gridpoints a discontinuity may fall. This may be quantified empirically in the case of the square microstructure. We perform an experiment in which we create data of square microstructure inclusions whose boundary falls exactly on the gridpoints. One dataset treats the boundary as open, and the other treats the boundary as closed; the input grid points that fall on the boundary differ between the two datasets. We quantify grid ambiguity error by the difference in the outputs of a model given both the open square data and the closed square data. We find that the absolute $H^1$ norm of the difference between these two outputs is $0.041$, which is slightly under twice the absolute $H^1$ norm of the output compared to the true solution, which has a value of $0.025$. We hypothesize that the model learns to assume the boundary falls near the middle of the grid square, which explains why the output difference between the two datasets is roughly twice the true error. From a theory standpoint, one could bound the Lipschitz constant of the FNO and compare it to the Lipschitz constant of the true map described by Proposition \ref{prop:stab_Lq}. However, we leave the theoretical estimates of error rates to future work.}

Finally, we compare the error in the effective $\overline{A}$ defined in \eqref{eq:cell}. This error is scaled by a difference between the Frobenius norms of the arithmetic and harmonic means of the true $A$ because the Frobenius norm of the true $\overline{A}$ should fall within that range. For this reason, in the case where the arithmetic and harmonic means are very close, as is frequently the case for the square and star inclusions, it is not valuable to learn the true $\overline{A}$. On the other hand, the varying-geometry Voronoi microstructure example on average has about 100 times greater difference between the arithmetic and harmonic means, in comparison with the star and square microstructure examples. This characteristic of the Voronoi microstructure further underscores the value of learning in this setting. 
\section{Conclusions}

In this work, we establish approximation theory for learning the
solution operator arising from the elliptic homogenization cell problem \eqref{eqn:cellprob}, viewed as a mapping from the coefficient to the solution; the theory allows for discontinuous
coefficients. We also perform numerical experiments that validate the theory, explore qualitative differences between various microstructures, and quantify error/cost trade-offs in
the approximation. We provide two  different stability results for the underlying solutions that build understanding of the underlying map. These stability results, when combined with existing universal approximation results for neural operators, result in rigorous approximation theory for learning in this problem setting. On the empirical side we provide, and then study numerically, examples of various microstructures that satisfy the conditions of the approximation theory. We observe that model error is dominated by error along discontinuous and corner interfaces, and that discontinuous microstructures give rise to qualitatively different learning behavior. Finally, we remark that the learned effective properties are highly accurate, especially in the case of the Voronoi microstructure that we regard as the most complex. Since discontinuous microstructures arise naturally in solid mechanics, understanding learning behavior in this context is an important prerequisite for using machine learning for applications. In this area and others, numerous questions remain which address the rigor necessary for use of machine learning in scientific applications.

We have confined our studies to one of the
canonical model problems of homogenization theory, the divergence form elliptic setting with periodic microstructure, to obtain deeper understanding of the learning constitutive laws. One interesting
potential extension of this work is the setting in which the material coefficient $A$ is not periodic
but random with respect to the microstructure. Another is where it is only locally periodic and has dependence on the macroscale variable as well; thus $A^{\epsilon} = A(x, \frac{x}{\eps})$. In this case, the form of the cell problem \eqref{eqn:cellprob} and homogenized coefficient \eqref{eq:cell} remain the same, but $A$ and $\chi$ both have parametric dependence on $x$. The approximation theory and the empirical learning problem would grow in complexity in comparison to what is developed here, but the resulting methodology could be useful
and foundational for understanding more complex constitutive models in which the force balance
equation couples to other variables. Indeed, the need for efficient learning of constitutive models is
particularly pressing in complex settings such as crystal plasticity. We anticipate
that the potential use of machine learning to determine parametric
dependence of constitutive models defined by homogenization will be for these more complex problems. The work described in this paper provides an underpinning conceptual approach, 
foundational analysis and set of numerical experiments that serve to underpin more applied work in this field.

\vspace{0.1in}


\vspace{0.1in}

\bibliographystyle{siam} 
\bibliography{ref}

\appendix
\addcontentsline{toc}{section}{Appendices}
\section*{Appendices}
\section{Proofs of Stability Estimates} \label{apdx:perturb}
In this section, we prove the stability estimates stated in Section \ref{ssec:SE}.
The following lemma is a modification of the standard estimate for parametric dependence
of elliptic equations on their coefficient. We include it here for completeness.

\Propstabinfty*

\begin{proof}
For existence and uniqueness of the solution to the cell problem using Lax-Milgram, we refer to the texts  \cite{blanc2023homogenization,pavliotis2008multiscale}; we simply derive the bounds and stability estimate. First, note that \eqref{eqn:cellprob} decouples, in particular,
\begin{equation}
\label{eq:decoupled_cell}
- \nabla \cdot (A \nabla \chi_{\ell})  = \nabla \cdot A e_{\ell}, \qquad y \in \Td
\end{equation}
for \(l=1,\dots,d\) where \(e_{\ell}\) is the \(\ell\)-th standard basis vector of \(\Rd\) and each \(\chi_{\ell} \in \dot{H}^1 (\mathbb{T}^d; \R)\). Multiplying by \(\chi_{\ell}\) and integrating by parts shows
\begin{align*}
    \alpha \|\nabla \chi_{\ell} \|_{L^2}^2 &\leq \int_{\Td} \langle A \nabla \chi_{\ell}, \nabla \chi_{\ell} \rangle \dy \\
    &= - \int_{\Td} \langle A e_{\ell}, \nabla \chi_{\ell} \rangle \dy \\
    &\leq \int_{\Td} |Ae_{\ell}| |\nabla \chi_{\ell}| \dy \\
    &\leq \left (\int_{\Td} |Ae_{\ell}|^2 \dy \right )^{\frac{1}{2}} \left (\int_{\Td} |\nabla \chi_{\ell} |^2 \dy \right )^{\frac{1}{2}} \\
    &\leq \|A\|_{L^\infty} \|\nabla \chi_{\ell} \|_{L^2}.
\end{align*}
Therefore
\[\|\nabla \chi \|_{L^2}^2 = \sum_{l=1}^d \|\nabla \chi_{\ell} \|_{L^2}^2 \leq \frac{d \|A\|_{L^\infty}^2}{\alpha^2}\leq \frac{d \beta^2}{\alpha^2},\]
which implies the first result.

To prove the second result, we denote the right hand side of \ref{eq:decoupled_cell} by $f^{(i)}_{\ell} = \nabla \cdot A^{(i)} e_{\ell}$ in what follows. For any $v \in \dot{H}^1(\Td;\R)$, we have that 
\begin{align*}
-\int_{\Td} \nabla \cdot (\Aone \nabla \chione_{\ell}) v \; \dy  & = \int_{\Td} \fone_{\ell} v\; \dy \\
-\int_{\partial \Td} v \Aone \nabla \chione_{\ell}\cdot \widehat{n} \; \dy + \int_{\Td }\nabla v \cdot  \Aone \nabla \chione_{\ell}\; \dy & = \int_{\Td} \fone_{\ell} v \; \dy.
\end{align*}

Since $v$, $\Aone$, and the solution $\chione_{\ell}$ are all periodic on $\Td$, the first term is $0$. Combining with the equation for $\chitwo_{\ell}$, we get
\begin{align*}
\int_{\Td}\nabla v \cdot \left(\Aone - \Atwo\right)&\nabla \chione_{\ell}  \; \dy  =\\&=\int_{\Td}(\fone_{\ell} - \ftwo_{\ell})v + \nabla v \cdot \left(\Atwo\left(\nabla \chitwo_{\ell} - \nabla \chione_{\ell}\right)\right) \dy.
\end{align*}
Setting $v = \chitwo_{\ell} - \chione_{\ell}$, we have
\begin{align*}
\int_{\Td}\left(\nabla \chitwo_{\ell} - \nabla \chione_{\ell}\right) &\cdot\left(\left(\Aone - \Atwo\right)\nabla\chione_{\ell}\right) \dy  = \int_{\Td} (\fone_{\ell} - \ftwo_{\ell})\left(\chitwo_{\ell} - \chione_{\ell}\right) \dy\\  &+ \int_{\Td} \left(\nabla \chitwo_{\ell} - \nabla \chione_{\ell}\right) \cdot \left(\Atwo \left(\nabla \chitwo_{\ell} - \nabla \chione_{\ell}\right)\right)  \dy, \\
\alpha \|\nabla \chitwo_{\ell} - \nabla\chione_{\ell}\|_{L^2}^2  &\leq \|\Aone - \Atwo\|_{L^{\infty}} \|\nabla \chione_{\ell}\|_{L^2}\|\nabla \chitwo_{\ell} - \nabla \chione_{\ell}\|_{L^2} \\&+ \|\fone_{\ell} - \ftwo_{\ell}\|_{\dot{H}^{-1}}\|\nabla \chitwo_{\ell} - \nabla \chione_{\ell}\|_{L^2},
\end{align*}
\begin{equation}\label{eqn:per_bnd_inf_f}
 \|\chitwo_{\ell} - \chione_{\ell}\|_{\dot{H}^1}  \leq \frac{1}{\alpha}\left(\|\Aone - \Atwo\|_{L^\infty}\|\nabla \chione_{\ell}\|_{L^2} + \|\fone_{\ell} - \ftwo_{\ell}\|_{\dot{H}^{-1}}\right).
\end{equation}
Evaluating,
\begin{align}
    \|\fone_{\ell} - \ftwo_{\ell}\|_{\dot{H}^{-1}} & = \|\nabla \cdot \Aone e_{\ell} - \nabla \cdot \Atwo  e_{\ell}\|_{\dot{H}^{-1}},\\
    & = \sup_{\|\xi\|_{\dot{H}^1}=1} \int_{\Td} \xi \nabla \cdot(\Aone - \Atwo)e_{\ell} \; \dy, \\
    & \leq \sup_{\|\xi\|_{\dot{H}^1}=1} \|(\Aone - \Atwo)e_{\ell}\|_{L^2} \|\nabla \xi \|_{L^2},\\
    & \leq \|\Aone - \Atwo\|_{L^2}\leq \|\Aone - \Atwo\|_{L^{\infty}}.
\end{align}
since our domain is $\Td$. Combining this with \eqref{eqn:per_bnd_inf_f} and the bound of $\|\nabla \chi_{\ell}\|_{L^2}\leq \frac{\beta}{\alpha}$ obtained in the first part of this proposition, we have 
\begin{equation}
    \|\chitwo_{\ell} - \chione_{\ell}\|_{\dot{H}^1}  \leq \frac{1}{\alpha}\left(1 + \frac{\beta}{\alpha}\right)\|\Aone- \Atwo\|_{L^\infty}.
\end{equation}
Returning to $d$ vector components yields the result. 
\end{proof}

The following result shows that the mapping \(A \mapsto \bar{A}\) is continuous on separable subspaces of $L^{\infty}(\Td;\Rdd)$.

\begin{lemma}
\label{lemma:extended_map_Linfty}
Let \(\A \subset L^\infty(\Td;\Rdd)\) be a separable subspace and \(K \subset \A \cap \PD\) a closed set in $L^{\infty}$. Define the mapping \(F : K \to \Rdd\) by \(A \mapsto \bar{A}\) as given by \eqref{eq:cell}. Then there exists a continuous mapping \(\mathcal{F} \in C(\A;\Rdd)\) such that \(\mathcal{F}(A) = F(A)\) for any \(A \in K\).
\end{lemma}

\begin{proof}
Let \(A^{(1)}, A^{(2)} \in K\) then, by Proposition~\ref{prop:stab_infty},
\begin{align*}
    \big | F \big( A^{(1)} \big ) &- F \big ( A^{(2)} \big ) \big |_F \leq \int_{\Td} |A^{(1)} - A^{(2)}|_F \big( 1 +  |\nabla \chi^{(1)} |_F \big ) \dy \\ &+ \int_{\Td} |A^{(2)}|_F |\nabla \chi^{(1)} - \nabla \chi^{(2)} |_F \dy \\
    &\leq \|A^{(1)} - A^{(2)}\|_{L^\infty} \big( 1 + \|\nabla \chi^{(1)}\|_{L^2} \big ) + \|A^{(2)}\|_{L^\infty} \|\nabla \chi^{(1)} - \nabla \chi^{(2)}\|_{L^2} \\
    &\leq \left ( 1 + \frac{\sqrt{d}}{\alpha} \left ( \|A^{(1)}\|_{L^\infty} + \|A^{(2)}\|_{L^\infty}  \left ( \frac{\min \big ( \|A^{(1)}\|_{L^\infty}, \|A^{(2)}\|_{L^\infty} \big )}{\alpha} + 1 \right ) \right )  \right ) \\
    &\;\;\;\; \cdot \|A^{(1)} - A^{(2)}\|_{L^\infty}
\end{align*}
hence \(F \in C(K;\Rdd)\). Applying the Tietze extension theorem \cite{dugundji1951extension} to \(F\) implies the existence of \(\mathcal{F}\).
\end{proof}

The following lemma shows that, unfortunately, separable subspaces of 

\noindent \(L^{\infty}(\Td;\Rdd)\) are not very useful. Indeed, in the desired area of application of continuum mechanics, we ought to be able to place a boundary of material discontinuity anywhere in the domain. The following result shows that doing so is impossible for a subset of $\PD$ which lies only in a separable subspace of $L^{\infty}(\Td; \Rdd)$.

\begin{lemma}
\label{lemma:no_separable_characteristics}
For any \(t \in [0,1]\) define \(c_t : [0,1] \to \R\) by 
\[c_t(x) = \begin{cases}
1, & x \leq t \\
0, & x > t
\end{cases}, \qquad \forall \: x \in [0,1].\]
Define \(E = \{c_t : t \in [0,1]\} \subset L^\infty([0,1])\). There exists no separable subspace \(\A \subset L^\infty([0,1])\) such that \(E \subseteq \A\).
\end{lemma}
\begin{proof}
Suppose otherwise. Since \((\A, \|\cdot\|_{L^\infty})\) is a separable metric space, \((E,\|\cdot\|_{L^\infty})\)  must be separable since \(E \subseteq \A\); this is a contradiction since \((E,\|\cdot\|_{L^\infty})\) is
not separable. To see this, let $\{c_{t_j}\}_{j=1}^\infty$ be an arbitrary countable susbset of \(E\). Then for any \(t \not \in \{t_j\}_{j=1}^\infty\), we have,
\[\inf_{\{t_j\}_{j=1}^\infty} \|c_t - c_{t_j}\|_{L^\infty} = 1.\]
Hence no countable subset can be dense.
\end{proof}

Instead of working on a compact subset of a separable subspace of $L^{\infty}(\Td; \Rdd)$, we may instead try to find a suitable probability measure which contains the discontinous functions of interest. The following remarks makes clear why such an approch would still be problematic for the purposes of approximation.
\begin{remark}[Gaussian Threshholding]
Let \(\mu\) be a Gaussian measure on 

\noindent\(L^2([0,1])\). Define
\[T(x) = \begin{cases}
1, & x \geq 0 \\
0, & x < 0
\end{cases}, \qquad \forall \: x \in [0,1]\]
and consider the corresponding Nemytskii operator \(N_T : L^2 ([0,1]) \to L^\infty ([0,1])\).
Then, working with the definitions in Lemma~\ref{lemma:no_separable_characteristics}, it is easy to see that 

\noindent\(E \subset \supp {N_T}^\sharp \mu\). Therefore there exists no separable subspace of \(L^\infty([0,1])\) which contains \(\supp {N_T}^\sharp \mu\).
\end{remark}

We therefore abandon $L^\infty$ and instead show continuity and Lipschitz continuity for some $L^q$ with $q < \infty$ to $\dot{H}^1$. The following lemma is a general result for convergence of sequences in metric spaces which is used in a more specific context in the next lemma. 
\begin{lemma}
\label{lemma:subsubsequences}
Let \((M,d)\) be a metric space and \((a_n) \subset M\) a sequence. If every subsequence \((a_{n_k}) \subset (a_n)\) contains a subsequence \((a_{n_{k_l}}) \subset (a_{n_k})\) such that \((a_{n_{k_l}}) \to a \in M\) then \((a_n) \to a\).
\end{lemma}
\begin{proof}
Suppose otherwise. Then, there exists some $\epsilon >0$ such that, for every $N \in \mathbb{Z}^+$, there exists some $n = n(N) > N$ such that 
$$d(a_n, a) \geq \epsilon.$$ 
Then we can construct a subsequence $(a_{n_j}) \subset (a_n)$ such that $d(a_{n_j},a) \geq \epsilon \;\forall n_j$. Therefore $a_{n_j}$ does not have a subsequence converging to $a$, which is a contradiction.
\end{proof}

The following lemma proves existence of a limit in $L^2(D;\Rd)$ of a sequence of outputs of operators in $L^{\infty}(D;\Rdd)$.
\begin{lemma}
\label{lemma:product_converges}
Let \(D \subseteq \Rd\) be an open set and \((A_n) \subset L^\infty(D;\Rdd)\) a sequence satisfying the following.
\begin{enumerate}
    \item $A_n \in \PD$ for all $n$, 
    \item There exists \(A \in L^\infty (D;\Rdd)\) such that
    \((A_n) \to A \text{ in } L^2(D;\Rdd).\)
\end{enumerate}
Then, for any \(g \in L^2(D;\Rd)\), we have that \((A_ng) \to Ag \) in \(L^2(D;\Rd)\). 
\end{lemma}
\begin{proof}
We have
\[\|A_ng\|_{L^2} \leq \beta \|g\|_{L^2}\]
hence \((A_n g) \subset L^2(D;\Rd)\) and, similarly, by finite-dimensional norm equivalence, there is a constant \(C_1 > 0\)
such that
\[\|Ag\|_{L^2} \leq C_1 \|A\|_{L^\infty} \|g\|_{L^2} \]
hence \(Ag \in L^2(D;\Rd)\). Again, by finite-dimensional norm equivalence, we have that there exists a constant \(C_2 > 0\)  
such that, for \(j \in \{1,\dots,d\}\) and almost every \(y \in D\), we have
\[(A_n g)_j(y)^2 \leq |A_n^{(j)}(y)|^2 |g(y)|^2 \leq C_2 \beta^2 |g(y)|^2\]
where \(A_n^{(j)}(y)\) denotes the \(j\)-th row of \(A_n^{(j)}(y)\). In particular,
\[|(A_n g)_j (y) | \leq \sqrt{C_2} \beta |g(y)|.\]
Let \((A_{n_k}) \subset (A_n)\) be an arbitrary subsequence. Since \((A_n) \to A\), we  have that \((A_{n_k}) \to A \) in \(L^2(D;\Rdd)\). Therefore, there exists a subsequence \((A_{n_{k_l}}) \subset (A_{n_k})\) such that \(A_{n_{k_l}}(y) \to A(y)\) for almost every \(y \in D\). Then \(A_{n_{k_l}}(y) g(y) \to A(y) g(y)\) for almost every \(y \in D\). Since \(|g| \in L^2(\Rd)\), we have, by the dominated convergence theorem, that \((A_{n_{k_l}} g)_j \to (Ag)_j\) in \(L^2(D)\) for every \(j \in \{1,\dots,d\}\). Therefore \((A_{n_{k_l}} g) \to Ag\) in \(L^2(D;\Rd)\). Since the subsequence \((A_{n_k})\) was arbitrary, Lemma~\ref{lemma:subsubsequences} implies the result.
\end{proof}

Finally, we may prove Proposition \ref{prop:cellproblem_cont_L2}. 

\PropstabLtwo*
\begin{proof}
Consider the PDE
\begin{equation}
\label{eq:decoupled_cell2}
- \nabla \cdot (A \nabla u) = \nabla \cdot A e, \qquad y \in \Td
\end{equation}
where \(e\) is some standard basis vector of \(\Rd\). Let \((A_n) \subset K\) be a sequence such that \((A_n) \to A \in K\) in  \(L^2(\Td;\Rdd)\). Denote by \(u_n \in \dot{H}^1 (\Td)\) the solution to \eqref{eq:decoupled_cell2} corresponding to each \(A_n\) and by \(u \in \dot{H}^1 (\Td)\) the solution corresponding to the limiting \(A\). A similar calculation as in the proof of Proposition \ref{prop:stab_infty} shows 
\begin{align*}
    \alpha \|u_n - u\|_{\dot{H}^1}^2 &\leq \int_{\Td} \langle (A - A_n)(\nabla u + e), \nabla u_n - \nabla u \rangle \dy \\
    &\leq \|u_n - u\|_{\dot{H}^1} \|(A_n - A)(\nabla u + e)\|_{L^2}.
\end{align*}
Since \(\nabla u + e \in L^2(\Td;\Rd)\), by Lemma~\ref{lemma:product_converges}, \(\big ( A_n(\nabla u + e) \big) \to A(\nabla u + e)\) in \(L^2(\Td;\Rd)\) hence \((u_n) \to u\) in \(\dot{H}^1(\Td)\). In particular, the mapping \(A \mapsto u\) defined by \eqref{eq:decoupled_cell2} is continuous. Since the problem \eqref{eqn:cellprob} decouples as shown by \eqref{eq:decoupled_cell}, we have that each component mapping \(G_l : K \to \dot{H}^1(\Td)\) defined by \(A \mapsto \chi_{\ell}\) is continuous thus \(G\) is continuous.
Applying the Tietze 
extension theorem \cite{dugundji1951extension} to \(G\) implies the existence of \(\G\).
\end{proof}

The following is a straightforward consequence of Proposition \ref{prop:cellproblem_cont_L2} that establishes continuity of the map $A \mapsto \overline{A}$ defined in \eqref{eq:cell} as well. 

\begin{lemma}
\label{lemma:extended_map_L2}
Endow \(\PD\) with the \(L^2(\Td;\Rdd)\) induced topology and let \(K \subset \PD\) be a closed set. Define the mapping \(F : K \to \Rdd\) by \(A \mapsto \bar{A}\) as given by \eqref{eq:cell}. Then there exists a bounded continuous mapping \(\mathcal{F} \in C\bigl(L^2(\Td;\Rdd);\Rdd\bigr)\) such that \(\mathcal{F}(A) = F(A)\) for any \(A \in K\).
\end{lemma}
\begin{proof}
Since \(\nabla : \dot{H}^1(\Td;\Rd) \to L^2(\Td;\Rdd)\) is a bounded operator, Lemma~\ref{prop:cellproblem_cont_L2} implies that the mapping \(A \mapsto A + A \nabla \chi^T\) is continuous as compositions, sums, and products of continuous functions are continuous. Now let \(A \in \PD\) then \(A \in L^1(\Td;\Rdd)\) since \(A \in L^\infty (\Td;\Rdd)\). Thus
\[\left | \int_{\Td} A \dy  \right |_F \leq \int_{\Td} |A|_F \dy \leq \|A\|_{L^2}  \]
by H{\"o}lder's inequality and the fact that \(\int_{\Td} \dy = 1\). Hence \(F \in C(K;\Rdd)\) as a composition of continuous maps. Again applying the Tietze 
extension theorem \cite{dugundji1951extension} to \(F\) implies the existence of \(\mathcal{F}\).
\end{proof}

To prove Proposition \ref{prop:stab_Lq}, we need to establish Lipschitz continuity. We first establish the following result, which is similar to the one proved in \cite{bonito2013adaptive} in Theorem 2.1. We show it again here both for completeness and because we specialize to the case of the cell problem \eqref{eqn:cellprob} with periodic boundary conditions rather than the system \eqref{eqn:ms_ellip} with Dirichlet boundary conditions.

\begin{lemma}\label{lem: Lpcont}
Let $A^{(1)}, A^{(2)} \in \PD$ and let $\chione, \chitwo$ be the corresponding solutions to \eqref{eqn:cellprob}.
Then
    \begin{equation}
    \|\chione - \chitwo \|_{\dot{H}^1} \leq \frac{\sqrt{d}}{\alpha}\left(\|\Atwo -\Aone\|_{L^2} + \|\nabla \chitwo\|_{L^p}\|\Atwo - \Aone\|_{L^q}\right)
    \end{equation}
    for $p \geq 2$ and $q = \frac{2p}{p-2}$.
\end{lemma}
\begin{proof}
As in the proof of Proposition \ref{prop:stab_infty}, we denote $f^{(i)} = \nabla \cdot A^{(i)}$ for $i \in \{1,2\}$ for simplicity of notation and to be easily comparable to the proof of Theorem 2.1 in \cite{bonito2013adaptive}. Since both sides of the cell problem equation \eqref{eqn:cellprob} depend on $A^{(i)}$, we introduce $\widetilde{\chi}$ as the solution of 
    \begin{align}
        -\nabla \cdot\left(\nabla \widetilde{\chi} \Atwo\right) & = \nabla \cdot \Aone, \quad \widetilde{\chi} \in \dot{H}^1(\Td;\Rd)
    \end{align}
    as an intermediate function. We obtain bounds using $\widetilde{\chi}$ and apply the triangle inequality to $$\|(\chione - \widetilde{\chi}) + (\widetilde{\chi} -\chitwo)\|_{\dot{H}^1}$$ to obtain a bound on $\|\chione - \chitwo\|_{\dot{H}^1}$. 
    From the na{\"i}ve perturbation bound in \eqref{eqn:per_bnd_inf_f} we have $$\|\widetilde{\chi}_{\ell} - \chitwo_{\ell}\|_{\dot{H}^1} \leq \frac{1}{\alpha}\|\fone_{\ell} - \ftwo_{\ell}\|_{\dot{H}^{-1}},$$ so we are left to bound $\|\chione_{\ell} - \widetilde{\chi}_{\ell}\|_{\dot{H}^1}$. We note that
    \begin{align*}
        \nabla \cdot \left(\Atwo \nabla \widetilde{\chi}_{\ell}\right) & = \nabla \cdot\left(\Aone \nabla\chione_{\ell}\right) \\
        \int_{\Td} \Atwo \nabla\widetilde{\chi}_{\ell} \cdot \nabla v \; \dy & = \int_{\Td} \Aone\nabla\chione_{\ell}\cdot \nabla v \dy \quad \forall v \in \dot{H}^1(\Td;\R)
    \end{align*}
    Letting $v = \chione_{\ell} - \widetilde{\chi}_{\ell}$, 
    \begin{align*}
        \int_{\Td} \Atwo \nabla\widetilde{\chi}_{\ell} \cdot \left(\nabla \chione_{\ell} - \nabla \widetilde{\chi}_{\ell}\right)\dy & = \int_{\Td}\Aone \nabla \chione_{\ell} \cdot \left(\nabla \chione_{\ell} - \nabla \widetilde{\chi}_{\ell}\right) \dy\\
        \int_{\Td}A^{(2)}\left(\nabla \widetilde{\chi}_{\ell} -\nabla \chione_{\ell}\right) &\cdot \left(\nabla \widetilde{\chi}_{\ell} - \nabla \chione_{\ell}\right) \dy \\ & = \int_{\Td}\left(\Atwo-\Aone\right)\nabla \chione_{\ell} \cdot \left(\nabla \chione_{\ell} - \nabla \widetilde{\chi}_{\ell}\right)\dy\\
        \alpha \|\widetilde{\chi}_{\ell} - \chione_{\ell}\|_{\dot{H}^1} & \leq \|(\Atwo - \Aone)(\nabla \chione_{\ell})\|_{L^2}
    \end{align*}
    Applying H{\"o}lder for $L^2$, we get 
    \begin{equation}
        \|\widetilde{\chi}_{\ell} - \chione_{\ell}\|_{\dot{H}^1} \leq \frac{1}{\alpha}\|\nabla \chione_{\ell}\|_{L^p}\|\Atwo - \Aone\|_{L^q}
    \end{equation}
    for $q = \frac{2p}{p-2}$ where $p \in [2,\infty]$.
    Putting the two parts together, we have that
    \begin{align*}
        \|\chitwo_{\ell} - \chione_{\ell}\|_{\dot{H}^1} &\leq \frac{1}{\alpha}\|\nabla\cdot \Atwo e_{\ell} -\nabla \cdot \Aone e_{\ell}\|_{\dot{H}^{-1}} +  \frac{1}{\alpha}\|\nabla \chione_{\ell}\|_{L^p}\|\Atwo - \Aone\|_{L^q}\\
        & \leq \frac{1}{\alpha}\|\Atwo - \Aone\|_{L^2} + \frac{1}{\alpha}\|\nabla \chione_{\ell}\|_{L^p}\|\Atwo - \Aone\|_{L^q}
    \end{align*}
    Combining bounds for all $d$ dimensions yields the result. 
\end{proof}
\begin{remark}
    Since $L^q(\Omega) \hookrightarrow L^2(\Omega)$ for bounded $\Omega \subset \R^d$ and $q \geq 2$, we could also write the bound of Lemma \ref{lem: Lpcont} as 
    \begin{equation*}
        \|\chitwo_{\ell} - \chione_{\ell}\|_{\dot{H}^1} \leq \frac{1}{\alpha}\left(C + \|\nabla \chione_{\ell}\|_{L^p}\right)\|\Atwo - \Aone\|_{L^q}
    \end{equation*}
    for some $C$ dependent only on $q$ and $\Omega$.
\end{remark}

The result of Lemma \ref{lem: Lpcont} is unhelpful if $\|\nabla \chi\|_{L^p}$ is unbounded. In this setting, it is not possible for Lemma \ref{lem: Lpcont} to result in Lipschitz continuity as a map from $L^2$ to $\dot{H}^1$. Instead, we seek to bound $\|\nabla \chi\|_{L^p}$ for some $p$ satisfying  $2 < p < \infty$.

Before continuing, we establish a bound on the gradient of the solution to the Poisson equation on the torus. This follows the strategy of \cite{bonito2013adaptive} for the Dirichlet problem. In order to avoid extra factors of $2\pi$ in all formulae, we work on the rescaled torus denoted $\Tdtwo = [0,2\pi]^d$ with opposite faces identified for the following result of Lemma \ref{lemma:reisz_transform}. As we work on the torus, it is useful to first set up notation for the function spaces of interest. Let 
\begin{equation*}
    \mathcal{D}(\Tdtwo) = C^{\infty}_c(\Tdtwo) = C^{\infty}(\Tdtwo)
\end{equation*}
be the space of test functions where the last equality follows from compactness of the torus. Functions can be either $\R$ or $\mathbb{C}$ valued hence we do not explicitly specify the range. We equip $\mathcal{D}(\Tdtwo)$ with a locally convex topology generated by an appropriate family of semi-norms, see, for example, \cite[Section 3.2.1]{triebel1987}. Any function $g \in \mathcal{D}(\Tdtwo)$ can be represented by its Fourier series
\begin{equation*}
    g(x) = \sum_{k \in \mathbb{Z}^d}\widehat{g}(k) \text{e}^{ix \cdot k}
\end{equation*}
where $\widehat{g}$ denotes the Fourier transform of $g$ and convergence of the right-hand side sum is with respect to the topology of $\mathcal{D}(\Tdtwo)$, and $i$ denotes the imaginary unit. It holds that $\widehat{g} \in \mathcal{S}(\mathbb{Z}^d)$, the Schwartz space of rapidly decreasing functions on the integer lattice, so we have 
\begin{equation*}
    |\widehat{g}(k)| \leq c_m (1 + |k|)^{-m}, \quad m = 0, 1, \dots
\end{equation*}
for some constants $c_m$. We may then define the topological (continuous) dual space of $\mathcal{D}(\Tdtwo)$, the space of distributions, denoted $\mathcal{D}'(\Tdtwo)$, which can be described as follows: the
condition that $f \in \mathcal{D}'(\Tdtwo)$ is characterized by the property
\begin{equation*}
    |\widehat{f}(k)| \leq b_m (1 + |k|)^m, \quad m = 0, 1, \dots
\end{equation*}
for some constants $b_m$. We take the weak${\text -}^*$ topology on $\mathcal{D}'(\Tdtwo)$ and generally use the prime notation for any such dual space. For any \(-\infty < s < \infty\), we define the fractional Laplacian as 
\begin{equation}\label{eqn:LaplaceFourier}
    (-\Delta)^s f = \sum_{k \in \mathbb{Z}^d\setminus\{0\}} |k|^{2s} \widehat{f}(k) \text{e}^{ik\cdot x}
\end{equation}
where the right-hand side sum converges in the topology of $\mathcal{D}'(\Tdtwo)$. It is easy to see that $(-\Delta)^s : \mathcal{D}'(\Tdtwo) \to \mathcal{D}'(\Tdtwo)$ is continuous. Furthermore, for any \(j \in \{1,\dots,d\}\), we define the family of operators $\tilde{R}_j : \mathcal{D}'(\Tdtwo) \to \mathcal{D}'(\Tdtwo) $, defining periodic Riesz transforms, by
\begin{equation}\label{eqn:RieszTransform}
\tilde{R}_j f = \sum_{k \in \mathbb{Z}^d} - \frac{i k_j}{|k|} \widehat{f}(k) \text{e}^{ik \cdot x}
\end{equation}
where we identify $\frac{k_j}{|k|}|_{k=0} = \lim_{|k|\to 0} \frac{k_j}{|k|} = 0$. Again, we stress that convergence of the right-hand side sum is in the topology of  $\mathcal{D}'(\Tdtwo)$. Lastly, we denote by $\mathcal{S}(\Rd)$ and $\mathcal{S}'(\Rd)$ the Schwartz space and the space of tempered distributions on $\Rd$ respectively; see, for example, \cite[Chapter 1]{stein1971introduction} for the precise definitions.

The following lemma establishes boundedness of the periodic Riesz transform on $L^p (\Tdtwo)$. It is essential in proving boundedness of the gradient to the solution of the Poisson equation on the torus. The result is essentially proven in \cite{stein1971introduction}. We include it here, in our specific torus setting, giving the full argument  for completeness.

\begin{lemma}
    \label{lemma:reisz_transform}
    There exists a constant \(c = c(d,p) > 0\) such that,
    for any \(j \in \{1,\dots,d\}\) and any $f \in L^p(\Tdtwo)$ for some $2 \leq p < \infty$, we have
    \[  \| \tilde{R}_j f \|_{L^p (\Tdtwo)} \leq c \|f\|_{L^p (\Tdtwo)}.\]
\end{lemma}
\begin{proof}
    Let \(g \in L^2 (\Rd) \cap L^p (\R^d)\) for some \(1 < p < \infty\). For any \(j \in \{1,\dots,d\}\), define the family of operators $R_j$ by
    \[(R_j g)(x) = \lim_{\delta^{-1}, \epsilon \to 0^+} \int_{\delta \geq |t| \geq \epsilon} g(x-t) K_j (t) \: dt,\]
    where
    \[K_j (t) = \frac{\Gamma \big ( (d+1)/2 \big ) t_j}{\pi^{(d+1)/2} |t|^{d+1}}\]
    and \(\Gamma\) denotes the Euler-Gamma function.
    By \cite[Chapter 4, Theorem 4.5]{stein1971introduction}, \(K_j \in \mathcal{S}'(\R^d)\) and its Fourier transform satisfies 
    \[\widehat{K}_j (t) = - \frac{i t_j}{|t|}.\]
 Therefore, for any \(\phi \in \mathcal{S}(\Rd)\), we have 
    \[(K_j * \phi)^{\widehat{}} \: (t) = - \frac{i t_j}{|t|} \widehat{\phi}(t)\]
    where \(*\) denotes convolution, see, for example, \cite[Chapter 1, Theorem 3.18] {stein1971introduction}. Since \(g \in L^2(\R^d)\), we therefore find that, by \cite[Chapter 6, Theorem 2.6]{stein1971introduction},
    \begin{equation}\label{eqn:RieszL2}
        (R_j g)^{\widehat{}} \: (x) = - \frac{i x_j}{|x|} \widehat{g}(x)
    \end{equation}
    for Lebesgue almost every \(x \in \R^d\). The result \cite[Chapter 6, Theorem 2.6]{stein1971introduction} further shows that there exists a constant \(c = c(d,p) > 0\) such that
    \[\|R_j g\|_{L^p (\Rd)} \leq c \|g\|_{L^p (\Rd)}.\]
    We note from \eqref{eqn:RieszL2} and the definition \eqref{eqn:RieszTransform} that $\tilde{R}_j$ may be viewed as $R_j$ with the restriction of the Fourier multiplier $ -\frac{i x_j}{|x|}$ to the lattice $\mathbb{Z}^d$. We can therefore use the transference theory of \cite{stein1971introduction} to establish boundedness of $\tilde{R}_j$ from the boundedness of $R_j$.
    In particular, note that the mapping \(x \mapsto -\frac{i x_j} {|x|}\) is continuous at all \(x \in \R^d\) except \(x=0\).  However, by symmetry, we have that, for all \(\epsilon > 0\)
    \[\int_{|x| \leq \epsilon} -\frac{i x_j}{|x|} \: dx = 0.\]
    Therefore we can apply \cite[Chapter 7, Theorem 3.8, Corollary 3.16]{stein1971introduction} to conclude that, since \(R_j\) is bounded from \(L^p (\Rd)\) to \(L^p(\Rd)\),  \(\tilde{R}_j\) is bounded from \(L^p(\Tdtwo)\) to \(L^p (\Tdtwo)\)
    with 
    \[\|\tilde{R}_j\|_{L^p (\Tdtwo) \to L^p (\Tdtwo)} \leq \|R_j\|_{L^p (\Rd) \to L^p (\Rd)}. \]
    This implies the desired result.
\end{proof}
We define the Bessel potential spaces by
\begin{equation*}
    L^{s,p} (\Tdtwo) = \{u \in \mathcal{D}'(\Tdtwo) \; | \; \|u\|_{L^{s,p} (\Tdtwo)} \coloneqq \|(I -\Delta)^{s/2} u\|_{L^p (\Tdtwo)} < \infty \}
\end{equation*}
for any \(- \infty < s < \infty\) and \(1 < p < \infty\). We also define the homogeneous version of these spaces, sometimes called the Riesz potential spaces, by
\begin{equation*}
    \dot{L}^{s,p} (\Tdtwo) = \{u \in \mathcal{D}'(\Tdtwo) \; | \; \|u\|_{\dot{L}^{s,p} (\Tdtwo)} \coloneqq \|(-\Delta)^{s/2} u\|_{L^p (\Tdtwo)} < \infty ,\; \int_{\Tdtwo} u \dy = 0 \}.
\end{equation*}
It is clear that $\dot{L}^{s,p} (\Tdtwo) \subset L^{s,p} (\Tdtwo)$ is closed subspace. We then have the following result for the Poisson equation.

\begin{lemma}\label{lem:condPLaplace}
    For each $f \in L^{s,p}(\Tdtwo)$, for $-\infty < s < \infty$ and $2 \leq p<\infty$, the solution $u$ of the equation 
    \begin{equation}\label{eqn:Laplace}
        - \Delta u = f, \quad u \text{ 1-periodic, } \int_{\Tdtwo} u \dy = 0
    \end{equation}
    satisfies 
    \begin{equation}
        \|\nabla u\|_{\dot{L}^{s+1,p} (\Tdtwo)} \leq K \|f\|_{\dot{L}^{s,p}(\Tdtwo)}
    \end{equation} 
    for some finite $K > 0$ depending only on $p$ and $d$.
\end{lemma}
\begin{proof}
From the definitions \eqref{eqn:LaplaceFourier} and \eqref{eqn:RieszTransform}, it is easy to see that the Riesz transform can be written as
\[\tilde{R}_j = - \partial_{x_j} (-\Delta)^{-1/2} \]
in the sense of distributions. Consider now equation \eqref{eqn:Laplace} with $f \in L^{s,p} (\Tdtwo)$ for $2 \leq p < \infty$. We have that
\begin{align*}
    \|\p_{x_j}u\|_{\dot{L}^{s+1,p} (\Tdtwo)} & = \|\p_{x_j}(-\Delta)^{-1}f\|_{\dot{L}^{s+1,p} (\Tdtwo)} \\
    & = \|\p_{x_j}(-\Delta)^{-1/2}(-\Delta)^{s/2}f\|_{L^p (\Tdtwo)}\\
    & = \|\tilde{R}_j (-\Delta)^{s/2}f\|_{L^p (\Tdtwo)}.
\end{align*}
It is clear that 
\[\|(-\Delta)^{s/2}f\|_{L^p (\Tdtwo)} = \|f\|_{\dot{L}^{s,p} (\Tdtwo)} < \infty\]
hence $(-\Delta)^{s/2}f \in L^p (\Tdtwo)$. We can thus apply Lemma~\ref{lemma:reisz_transform} to find a constant $c = c(d,p) > 0$ such that 
\[\|\p_{x_j}u\|_{\dot{L}^{s+1,p} (\Tdtwo)} \leq c \|(-\Delta)^{s/2}f\|_{L^p (\Tdtwo)} = c \|f\|_{\dot{L}^{s,p} (\Tdtwo)}.\]
The result follows by finite-dimensional norm equivalence. 
    
\end{proof}

Next we define the homogeneous Sobolev spaces on the torus as 
\begin{equation}
    \dot{W}^{k,p}(\Td) = \{u \in W^{k,p}(\Td) \; | \; u \text{ is } 1 \text{-periodic}, \; \int_{\Td} u \dy = 0 \}
\end{equation}
for $k=0,1,\dots,$ and $1 \leq p \leq \infty$ with the standard norm on $W^{k,p}$, see, for example \cite{adams2003sobolev}. 

\begin{remark}
\label{remark:sobolev_duals}
By \cite[Section 3.5.4]{triebel1987}, we have that, for any \(k = 0,1,\dots\) and $1 < p < \infty$,
\[L^{k,p} (\Td) = W^{k,p} (\Td), \qquad\dot{L}^{k,p} (\Td) = \dot{W}^{k,p} (\Td).\]
Furthermore, by \cite[Section 3.5.6]{triebel1987},
\begin{align*}
    W^{-k,p'} (\Td) &= \big( W^{k,p} (\Td) \big )' = \big ( L^{k,p} (\Td) \big )' =  L^{-k,p'} (\Td), \\
    \dot{W}^{-k,p'} (\Td) &= \big( \dot{W}^{k,p} (\Td) \big )' = \big ( \dot{L}^{k,p} (\Td) \big )' =  \dot{L}^{-k,p'} (\Td)
\end{align*}
where \(p'\) is the H{\"o}lder conjugate of \(p\) i.e. \(1/p + 1/p' = 1\).
\end{remark}





In the following, we use the notation 
\begin{equation}
    [K_0,K_1]_{\theta, q}
\end{equation}
to denote the real interpolation between two Banach spaces continuously embedded in the same Hausdorff topological space, as described in \cite{adams2003sobolev}. We also need Lemma A1 from \cite{guermond2009lbb}, which we have copied below as Lemma \ref{lem:guermond}
to ease readability. Although this lemma was written only for $q=2$, the result still holds for our $q > 2$ with a very similar proof.

\begin{lemma}\label{lem:guermond}
    Let $E_1 \subset E_0$ be two Banach spaces with $E_1$ continuously embedded in $E_0$. Let $T: E_j \to E_j$ be a bounded operator with closed range and assume that $T$ is a projection, $j \in \{0,1\}$. Denote by $K_0$ and $K_1$ the ranges of $T|_{E_0}$ and $T|_{E_1}$ respectively. Then the following two spaces coincide with equivalent norms: 
    \[ [K_0, K_1]_{\theta, q} = [E_0, E_1]_{\theta, q} \cap K_0 \quad \forall \theta \in (0,1).\]
\end{lemma}

We now state the result for the bound on $\|\nabla \chi\|_{L^p}$ with a proof largely developed in \cite{bonito2013adaptive}.

\begin{lemma}\label{condP} Let $\chi$ solve \eqref{eqn:cellprob} for $A \in \PD$. Then
    \begin{equation}
        \|\nabla \chi\|_{L^p} \leq \frac{K^{\eta(p)}}{1-K^{\eta(p)}\left(1-\frac{\alpha}{\beta}\right)}
    \end{equation}
    for $2 \leq p < p^*\left(\frac{\alpha}{\beta}\right)$ where 
    \begin{equation}
        p^*(t) : = \text{max}\left\{p \;|\; K^{-\eta(p)} \geq 1-t, \; 2 < p < Q\right\}
    \end{equation} for $\eta(p) = \frac{ 1/2 - 1/p}{1/2 - 1/Q}$and $K= K(d,Q)$ is the constant in Lemma \ref{lem:condPLaplace}, for any choice of $Q>p$. 
\end{lemma}
\begin{proof}
    The operator $ T = -\Delta$ is invertible from \mt{$H^{-1}$} to $\dot{H}^1$, and the inverse $T^{-1}$ is bounded with norm $1$ since the Poisson equation with periodic boundary conditions has a unique solution in $\dot{H}^1$ for \mt{$f \in H^{-1}$} with bound \mt{$\|u\|_{\dot{H}^1} \leq \|f\|_{H^{-1}}$.} From Lemma \ref{lem:condPLaplace} it is also bounded with norm $K=K(d,Q)$ from \mt{$W^{-1,Q}$} to $\dot{W}^{1,Q}$ for any $Q > 2$. By the real method of interpolation \cite{adams2003sobolev}, for $2<p<Q$ we have that \begin{equation}\label{eq:realinterp}
        W^{1,p} = \left[H^1, W^{1,Q}\right]_{\eta(p), p}
    \end{equation}
    using the notation of \cite{adams2003sobolev} where $\eta(p) = \frac{1/2 - 1/p}{1/2 - 1/Q}$. From the duality theorem (Theorem 3.7.1. of \cite{bergh2012interpolation}), we have that 
    \begin{equation}
        \left[H^{-1}, W^{-1,Q}\right]_{\eta(p),p} = \left(\left[H^1,W^{1,Q'}\right]_{\eta(p),p'}\right)'
    \end{equation}
    From real interpolation, the right hand side equals $(W^{1,p'})' = W^{-1,p}$ in our notation. Therefore, we have the necessary dual statement that parallels \eqref{eq:realinterp}: 
    \begin{equation}\label{eqn:dualrealinterp}
        W^{-1,p} = \left[H^{-1},W^{-1,Q}\right]_{\eta(p),p}.
    \end{equation}
    Next we restrict these spaces to functions with periodic boundary conditions. Using the projection onto the space of continuous, periodic functions on $\Td$ and noticing that $W^{1,Q} \hookrightarrow H^1$, we apply Lemma \ref{lem:guermond} with $K_0 = \dot{H}^1$ and have 
    \begin{equation}\label{eqn:realinterprestricted}
        \dot{W}^{1,p} = [\dot{H}^1, \dot{W}^{1,Q}]_{\eta(p), p}.
    \end{equation} 
    Using the exact interpolation theorem, Theorem 7.23 of \cite{adams2003sobolev}, $T^{-1}$ is also a bounded map from \mt{$W^{-1,p}$} to $\dot{W}^{1,p}$ with norm $K^{\eta(p)}$: 
    \begin{equation}
        \|T^{-1}f\|_{\dot{W}^{1,p}}  \leq K^{\eta(p)}\|f\|_{W^{-1,p}}.
    \end{equation}
    The remainder of the proof is identical to that of the proof of Proposition $1$ in \cite{bonito2013adaptive}, but we state it here in our notation for completeness. Define \mt{$S$: $\dot{W}^{1,p} \to W^{-1,p}$} as the operator $Su = -\nabla \cdot \left(\frac{1}{\beta}A\nabla u\right) $. Let $V$ be the perturbation operator $V: = T - S$. Since $A \in \PD$, $S$ and $V$ are bounded operators from $\dot{W}^{1,p}$ to \mt{$W^{-1,p}$}, with the operator norms $\|S\|\leq 1$ and $\|V\| \leq 1 - \frac{\alpha}{\beta}$. Therefore, 
    \begin{equation}
        \|T^{-1}V\|_{\dot{W}^{1,p}\to \dot{W}^{1,p}} \leq \|T^{-1}\|_{W^{-1,p} \to \dot{W}^{1,p}}\|V\|_{\dot{W}^{1,p}\to W^{-1,p}} \leq K^{\eta(p)}\left(1 - \frac{\alpha}{\beta}\right),
    \end{equation}
    \mt{where the input and output spaces defining the operator norms are included for clarity.}
    Since $T$ is invertible, $S = T(I - T^{-1}V)$ is invertible provided $K^{\eta(p)}\left(1- \frac{\alpha}{\beta}\right) < 1$. Moreover, for $S^{-1}$ as a mapping from \mt{$W^{-1, p}$} to $\dot{W}^{1,p}$, 
    \begin{equation}
        \|S^{-1}\|  \leq \|(I-T^{-1}V)^{-1}\|_{\dot{W}^{1,p} \to \dot{W}^{1,p}}\|T^{-1}\|_{W^{-1,p}\to \dot{W}^{1,p}} \leq \frac{K^{\eta(p)}}{1-K^{\eta(p)}\left(1-\frac{\alpha}{\beta}\right)}.
    \end{equation}
    Therefore, \mt{
    \begin{equation}
        \|\nabla \chi \|_{L^p} = \|\chi\|_{\dot{W}^{1,p}} \leq \frac{1}{\beta}\|S^{-1}\| \|\nabla \cdot A\| \leq  \frac{K^{\eta(p)}}{1-K^{\eta(p)}\left(1-\frac{\alpha}{\beta}\right)}
    \end{equation}}
    provided $K^{\eta(p)}\left(1- \frac{\alpha}{\beta}\right) < 1$. The bound and specified range of $p$ follow. 
\end{proof}

Finally, we may prove Proposition \ref{prop:stab_Lq} 

\PropstabLq*
\begin{proof}
    Lemma \ref{condP} guarantees a $p_0> 2$ such that $\|\nabla \chi^{(2)}\|_{L^p}$ in Lemma \ref{lem: Lpcont} is bounded above by a constant for $2 < p< p_0$. Then Lemma \ref{lem: Lpcont} gives Lipschitz continuity of the solution map from $L^q(\Td) \mapsto \dot{H}^1(\Td)$ for $q$ satisfying $q_0 < q < \infty$ for some $q_0 > 2$. 
\end{proof}
\begin{remark}\label{rem:q0}
    From the results of Lemma \ref{condP} and Lemma \ref{lem: Lpcont}, we have that we can take $q_0 = \frac{2p_0}{p_0 - 2}$ where 
    \begin{equation*}
        p_0 = \max \{p \; | \; K^{-\eta(p)} \geq 1-t, \; 2 < p < Q\}.
    \end{equation*}
    Therefore, bounds on $p_0$ may be inherited from bounds on $K$ that appears in Lemma \ref{lem:condPLaplace}.
\end{remark}

\section{Proofs of Approximation Theorems} \label{apdx:approxthms}
In this section we prove the approximation theorems stated in Section \ref{sec:ApproxThms}. 

\thmLtwochi*
\begin{proof}
By Proposition~\ref{prop:cellproblem_cont_L2}, there exists a continuous map 

\noindent\(\G \in C(L^2(\Td;\Rdd);\dot{H}^1(\Td;\Rd))\) such that \(\G(A) = G(A)\) for any \(A \in K\). By \cite[Theorem 5]{kovachki2021universal}, there exists a FNO \(\Psi: L^2(\Td;\Rdd) \to \dot{H}^1(\Td;\Rd)\) such that 
\[\sup_{A \in K} \|\G(A) - \Psi(A) \|_{\dot{H}^1} < \epsilon.\]
Therefore
\[\sup_{A \in K} \|G(A) - \Psi(A)\|_{\dot{H}^1} = \sup_{A \in K} \|\G(A) - \Psi(A)\|_{\dot{H}^1} < \epsilon\]
as desired.
\end{proof}
\thmLtwoAbar*
\begin{proof}
The result follows as in Theorem~\ref{thm:fno_cell} by applying Lemma~\ref{lemma:extended_map_L2} instead of Proposition~\ref{prop:cellproblem_cont_L2}.
\end{proof}

\section{Proofs for Microstructure Examples}\label{apdx: micro}

\mt{
The following lemma establishes the compactness of subsets of $\PD$
generated by the probability measures from Section \ref{sec:E}.
As we are unaware of a proof in the literature, we have provided one below. The proof uses the $L^1$-Lipschitz spaces, which are defined as
\[\Lip_{\alpha}(L^1) = \{u \in L^1: \; \exists M(u) > 0: \; \omega(u,t)_1 \leq M t^{\alpha}\}\]
where $\omega(u,t)_1$ is the $1$-modulus of continuity, defined via
\[\omega(u,t)_1 = \sup_{0 \leq h \leq t} \|\tau_h u - u \|_{L^1(\Td)}.\]
\begin{lemma} \label{lem:BVLinfty_L2compact}
    $\BV(\Td) \cap L^{\infty}(\Td)$ is compactly embedded in $L^2(\Td)$.
\end{lemma}
\begin{proof}
    Let $u \in B$, where $B$ is a bounded subset of $\BV(\Td) \cap L^{\infty}(\Td)$ with $L^{\infty}$ norm and $BV$ seminorm bounded by $M$, and let $\tau_h f$ denote the translation of $f$ by $h$, i.e. $\tau_h f(x) = f(x-h)$. Then
    \begin{align}
        \|\tau_h u - u\|_{L^2} & \leq \|\tau_h u - u\|_{L^1}^{1/2}\|\tau_h u - u \|_{L^{\infty}}^{1/2}.
    \end{align}
    Since $\BV(\Td) \equiv \Lip_1(L^1(\Td))$, $\|\tau_h u - u \|_{L^1} \leq \|u\|_{\BV}|h|$. We have then 
    \begin{align*}
        \|\tau_h u - u \|_{L^2} & \leq \|u\|_{\BV}^{1/2}|h|^{1/2}(2M)^{1/2}.
    \end{align*}
    By the Fr\'echet-Kolmogorov theorem \cite{yosida2012functional}, this equicontinuity result is sufficient for compactness of $B$ in $L^2(\Td)$.
\end{proof}

Using the result of Lemma \ref{lem:BVLinfty_L2compact}, we see that any set of microstructure coefficients bounded in $L^{\infty}(\Td) \cap BV(\Td)$ satisfies the compactness assumption of the Approximation Theorems in Section \ref{sec:ApproxThms}. It is clear that the method of construction of the examples in Subsection \ref{ssec:micro_imple} leads to such sets. }
\mt{
\section{Numerical Implementation Details} \label{apd:implementation}
All FNO models are implemented in pytorch using python 3.9.7. Unless otherwise specified, the models have $18$ modes in each dimension, a width of $64$, and $4$ hidden layers. The lifting layer is a linear transformation with trainable parameters, and the projecting layer is a pointwise multilayer perceptron with trainable parameters. The batch size is $20$, the learning rate is $0.001$, and the number of epochs is $400$. These hyperparameters are chosen with a small grid search, but we emphasize that the FNO does not drastically change in performance unless these parameters are changed by an order of magnitude. For a model trained on $9500$ data using these hyperparameters and accelerated with an Nvidia P100 GPU, the training time is approximately $7$ hours. In Figures \ref{fig:disc-err}, \ref{fig:data-size-compare}, and \ref{fig:err-vs-size}, the error bars shown correspond to two standard deviations in each direction over the five samples. All code for this work may be found at \href{https://github.com/mtrautner/LearningHomogenization/}{github.com/mtrautner/LearningHomogenization/}.}

\end{document}